\documentclass[11.5pt]{article}
\usepackage[centertags]{amsmath}
\usepackage{amsfonts}
\usepackage{amssymb,amsmath,amsfonts}
\usepackage{booktabs}
\usepackage{array, color}
\usepackage{placeins}
\usepackage{float}
\usepackage{relsize}
\usepackage{lipsum}
\usepackage{graphicx}
\usepackage{subcaption}
\usepackage{float}
\restylefloat{table}
\usepackage{cite}
\usepackage{derivative}
\usepackage{multirow}
\usepackage{amssymb,amsmath,amsfonts,fancyhdr}
\usepackage{amsmath}
\usepackage{bigints}
\numberwithin{equation}{section}
\usepackage{amsthm}
\usepackage{mathtools}

\usepackage[utf8]{inputenc}
\usepackage[english]{babel}
\newtheorem{theorem}{Theorem}[section]
\newtheorem{corollary}{Corollary}[theorem]
\newtheorem{remark}{Remark}[theorem]
\newtheorem{lemma}{Lemma}[section]
\usepackage{bm}
\usepackage{amssymb}

\usepackage{booktabs}
\usepackage{color}
\usepackage{subfloat}
\usepackage{hyperref}
\usepackage{multirow}

\usepackage{lscape}
\usepackage{natbib}
\makeatletter
\newcommand{\thickhline}{%
    \noalign {\ifnum 0=`}\fi \hrule height 1pt
    \futurelet \reserved@a \@xhline
}
\newcolumntype{"}{@{\hskip\tabcolsep\vrule width 1pt\hskip\tabcolsep}}
\makeatother
\makeatletter \oddsidemargin  -.1in \evensidemargin -.1in
\newcommand{\doublespacing}{\let\CS=\@currsize
 \renewcommand{\baselinestretch}{1.05}\tiny\CS}
\date{ }
\begin{document}
\newcommand{\bea}{\begin{eqnarray}}
\newcommand{\eea}{\end{eqnarray}}
\newcommand{\nn}{\nonumber}
\newcommand{\bee}{\begin{eqnarray*}}
\newcommand{\eee}{\end{eqnarray*}}
\newcommand{\lb}{\label}
\newcommand{\nii}{\noindent}
\newcommand{\ii}{\indent}
\newtheorem{thm}{Theorem}[section]
\newtheorem{lem}{Lemma}[section]
\newtheorem{rem}{Remark}[section]
\theoremstyle{remark}
\renewcommand{\theequation}{\thesection.\arabic{equation}}
\vspace{5cm}
\title{\bf Estimation of the Selected Treatment Mean in Two Stage Drop-the-Losers Design}
\author{Masihuddin$^{1}$\thanks{Email : masih.iitk@gmail.com, masihst@iitk.ac.in}
	~~	
	and~ Neeraj Misra$^{2}$\thanks{Email : neeraj@iitk.ac.in}}
\maketitle \noindent {\it $^{1,2}$ Department of Mathematics \& Statistics, Indian
Institute of Technology Kanpur, Kanpur-208016, Uttar Pradesh, India} 

\newcommand{\oddhead}{ Estimation in Two-Stage DLD}
\renewcommand{\@oddhead}
{\hspace*{-3pt}\raisebox{-3pt}[\headheight][0pt]
{\vbox{\hbox to \textwidth
{\hfill\oddhead}\vskip8pt}}}
\vspace*{0.05in}
\noindent {\bf Abstract}: 
A common problem faced in clinical studies is that of estimating the effect of the most effective (e.g., the one having the largest mean) treatment among $k~(\geq2)$ available treatments. The most effective treatment is adjudged based on numerical values of some statistic corresponding to the $k$ treatments. A proper design for such problems is the so-called ``Drop-the-Losers Design (DLD)". We consider two treatments whose effects are described by independent Gaussian distributions having different unknown means and a common known variance. To select the more effective treatment, the two treatments are independently administered to $n_1$ subjects each and the treatment corresponding to the larger sample mean is selected. To study the effect of the adjudged more effective treatment (i.e., estimating its mean), we consider the two-stage DLD in which $n_2$ subjects are further administered the adjudged more effective treatment in the second stage of the design. We obtain some admissibility and minimaxity results for estimating the mean effect of the adjudged more effective treatment. The maximum likelihood estimator is shown to be minimax and admissible. We show that the uniformly minimum variance conditionally unbiased estimator (UMVCUE) of the selected treatment mean is inadmissible and obtain an improved estimator. In this process, we also derive a sufficient condition for inadmissibility of an arbitrary location and permutation equivariant estimator and provide dominating estimators in cases, where this sufficient condition is satisfied. The mean squared error and the bias performances of various competing estimators are compared via a simulation study. A real data example is also provided for illustration purpose.  \vspace{2mm}\\
\noindent {\it AMS 2010 SUBJECT CLASSIFICATIONS:} 62F07 · 62F10 · 62C20\\

\newcommand\blfootnote[1]{%
	\begingroup
	\renewcommand\thefootnote{}\footnote{#1}%
	\addtocounter{footnote}{-1}%
	\endgroup
}\blfootnote{
}
\section{Introduction}
In clinical and drug development studies, while evaluating the effects of multiple treatments (or drugs), a healthcare practitioner is often interested in choosing the most efficacious treatment out of $k ~(\geq 2)$ active treatments. The effectiveness of a treatment is evaluated on the basis of the characteristic (or a parametric function) associated with it. A treatment's mean effect is usually a primary characteristic of interest. 
Statistical significance tests (such as the test of homogeneity), which investigate the hypothesis of equal treatment effects, are the traditional statistical tool to such a problem. We conclude that the treatment effects aren't equal if this hypothesis is rejected, but in that case we don't know which treatment is the best (more effective). An obvious choice will be to select the treatment that corresponds to the largest mean. Such a strategy is known to be optimal under quite general scenarios (see, for example \cite{bahadur1952impartial}, \cite{eaton1967some} and \cite{misra1994non}). After choosing one of the treatments/drugs as the most effective one, an interesting subject is that of estimating the mean effect of the chosen treatment  where some additional data is available at the time of estimation.
In the biomedical field, this need has resulted in the emergence of a statistical framework, called as the  adaptive clinical trials.
 An adaptive design typically consists of several phases. Data analyses are performed at each stage, and adaptations are made based on updated information to achieve maximum  likelihood of success. For a detailed informational review on developments in inference procedures for a two-stage adaptive trial setting, one may refer to \cite{bauer2016twenty} and \cite{robertson2021point}. \par
A drop-the-loser design (DLD) is a multistage adaptive design, under which interim analyses are performed at each stage, and losers (e.g., inferior treatment groups) are dropped based on predetermined criteria. 
For an introduction of two-stage drop-the-losers methodology see \cite{thall1988two} and \cite{thall1989two}. For further developments one may refer to  \cite{bauer1999combining}, \cite{sampson2005drop}, \cite{wu2010interval}  and \cite{lu2013estimating}. \par
\cite{cohen1989two} first proposed an unbiased estimate for the mean of the best performing treatment in a purely statistical set up using a two stage adaptive design framework. They have derived an uniformly minimum variance conditionally unbiased estimator (UMVCUE) for the best performing candidate in the Gaussian setting with a common known and unknown variance. However, the proposed estimator is only meaningful for the case when the stage 1 sample sizes are all equal. For the case of known variance,  \cite{bowden2008unbiased} extended the method of \cite{cohen1989two} to unequal stage one and stage two sample sizes. They proposed a two-stage unbiased estimate for the best performing treatment at the interim. Their method is valid when the quantity of interest is the best, second best, or the $j^{th}$ best treatment among $k$ treatments. In the Gaussian setting, for the case of common unknown variance, \cite{robertson2019conditionally} provided an UMVCUE which works for multiple selected candidates, as well as unequal stage one and stage two sample sizes. \cite{sill2007extension} extended the methodology of \cite{cohen1989two} to the bivariate normal case where the covariance matrix is assumed to be known.  \par 

In the literature, most of the work in two stage adaptive design setting is either in hypothesis testing framework or on investigating estimators for the selected treatment mean when bias is the main criterion. In a two-stage context, however, proposing efficient estimators when mean squared error is the key criterion, has received relatively little attention in the literature. \cite{lu2013estimating} considered the problem of estimating the mean of the selected normal population in two-stage adaptive design and have shown that for $k \geq 3$ the naive estimator, which is the weighted average of the first and second stage sample means (with weights being proportional to the corresponding sample sizes)  is not minimax under squared error and LINEX loss functions. In this article, building upon the work of \cite{lu2013estimating}, we consider the problem of estimating the mean of the selected treatment in a two stage drop-the-losers setting and obtain various decision theoretic results under the criterion of mean squared error. \par
The paper is organized as follows: In Section 2, we introduce some notations and preliminaries that will be used throughout the paper. In Section 3, we prove that the naive estimator, which is the weighted average of the first and second stage sample means of adjudged effective treatment, is  minimax and admissible for estimating the selected treatment mean. In Section 4, we derive a sufficient condition for inadmissibility of an arbitrary location and permutation equivariant estimator. Some additional estimators, having nice performances in terms of mean squared error, are also derived in Section 4 of the paper. In section 5 , we report a simulation study to numerically assess performances of the various competing estimators under the criterion of mean squared error and Bias. In Section 5, we present a data analysis to illustrate implementation of proposed methodology.


\section{Preliminaries}
We consider a Drop-the-Losers design which is divided into two stages, with a single data-driven selection made in the middle. In the first stage of the design, referred to as Stage A, the investigators would independently administer two treatments (say, $T_1$ and $T_2$)  to $n_1$ subjects each. Suppose that the effects of two treatments $T_1$ and $T_2$ are characterized by two Gaussian populations $N\left(\mu_i, \sigma^2 \right),~ i=1,2$, having different unknown means ($\mu_i \in (-\infty,\infty),~i=1,2$) and a common known variance $\sigma^2$($> 0$). The treatment corresponding to $\max\{\mu_{1},\mu_{2}\}$ is considered to be the promising treatment or better treatment. Let $\overline{X}_i,~ i=1,2,$ be  the sample means (treatment effect estimates) corresponding to the two populations. Consider the natural selection rule that selects the treatment yielding the larger sample mean as the better treatment. Let $S \in \{1,2\}$ denote the index of the selected treatment $T_S$ (i.e. $S=1$, if $\overline{X}_1 > \overline{X}_2$,  $S=2$, if $\overline{X}_2 \geq \overline{X}_1$). $T_S$ is then taken forward in the second stage, referred to as Stage B of the two stage trial, for further confirmatory analysis. We draw another sample of size $n_2$ independently for the treatment $T_S$ selected in Stage A. Let $\overline{Y}$ be the stage B sample mean from the selected treatment $T_S$. The goal is to estimate $\mu_{S}$, the treatment effect of the selected treatment. Clearly, $\overline{X}_i \sim N\left(\mu_i, \frac{\sigma^2}{n_1} \right),~i=1,2,$ are independently distributed  and, conditioned on $S$, $\overline{Y} \sim N\left(\mu_S, \frac{\sigma^2}{n_2} \right)$. \par 
 We will be using the following notations throughout the paper:\\
\begin{itemize}
	\item[ $\bullet$] $\mathbb{R}$: the real line $\left(-\infty,\infty\right)$;
	\item[ $\bullet$] $\mathbb{R}^k$: the $k$ dimensional Euclidean space, $k \in \{2,3,\ldots\}$;
	\item[ $\bullet$] $N\left(\lambda, \tau^2 \right)$: normal distribution with mean $\lambda \in \mathbb{R}$ and standard deviation $\tau \in \left(0,\infty\right) $;
	\item[ $\bullet$] $N_2\left(\underline{\mu}, \Sigma \right)$: the bivariate normal distribution with mean vector $\underline{\mu} \in \mathbb{R}^2$ and positive definite covariance matrix $\Sigma$;
	\item[ $\bullet$] $\phi(\cdot) $: probability density function (p.d.f.) of $N(0,1)$;
	\item[ $\bullet$] $\Phi(\cdot)$: cumulative distribution function (c.d.f.) of $N(0,1)$;
	\item[ $\bullet$] For real numbers  $a$ and $b$
	\begin{equation*}
	\label{S1.E1}
	I(a\geq b)=\begin{cases}
	1, & \text{if $a \geq b$ }\\
	0, & \text{if $a < b$}.
	\end{cases}
	\end{equation*}
\end{itemize}
 \par

In addition to the notations introduced above, we make use of the following notations throughout the paper:\\
\noindent $\underline{X} = (\overline{X}_1,\overline{X}_2, \overline{Y})$; $\underline{\mu} = (\mu_1,\mu_2)$; $\Theta = \mathbb{R}^2$;  $\overline{X}_{S} = \max (\overline{X}_1 ,\overline{X}_2)$ (maximum of $\overline{X}_1$ and $\overline{X}_2$); $\overline{X}_{3-S} = \min (\overline{X}_1 ,\overline{X}_2)$ (minimum of $\overline{X}_1$ and $\overline{X}_2$);  $D_1 = \overline{X}_{3-S} - \overline{X}_{S}$;  $D_2 =  \overline{Y} - \overline{X}_{S}$; $\theta_1 = \min(\mu_1, \mu_2)$;  $\theta_2 = \max(\mu_1, \mu_2)$;  $\theta = \theta_2 - \theta_1$.  Also, for any $\underline{\mu} \in \Theta$, $\mathbb{P}_{\underline{\mu}}(\cdot)$ will denote the probability  measure induced by $\underline{X} = (\overline{X}_1,\overline{X}_2, \overline{Y})$, when $\underline{\mu} \in \Theta$ is the true parameter value, and $\mathbb{E}_{\underline{\mu}}(\cdot)$ will denote the expectation operator under the probability measure $\mathbb{P}_{\mu}(\cdot)$, $\underline{\mu} \in \Theta$. Note  that $\theta \geq 0$ and $\mathbb{P}_{\underline{\mu}}(D_1\leq0)=1, \forall~ \underline{\mu} \in  \Theta$.\par
We consider estimation of 
\begin{equation}
\label{S1.E1}
\mu_S=\begin{cases}
\mu_1, & \text{if $\overline{X}_1  \geq \overline{X}_2$ }\\
\mu_2, & \text{if $\overline{X}_1 < \overline{X}_2$}
\end{cases},
\end{equation}
under the squared error loss function
\begin{equation}
\label{S2.E2}
L_{\underline{T}}(\underline{\theta},a) = \left( a-\mu_S\right)^2  ,~~ \underline{\theta} \in \Theta, ~ a \in \mathcal{A} = \mathbb{R}.
\end{equation}
Under the squared error loss function $(2.2)$, the risk function (also referred to as the  mean squared error) of an estimator $\delta(\cdot)$ is defined by
$$ R(\underline{\mu}, \delta)=\mathbb{E}_{\underline{\mu}}((\delta(\underline{X})-\mu_S)^2),~ \underline{\mu} \in \Theta.$$ 
An estimator $\delta(\underline{X})$ is said to be conditionally unbiased if  for every $\underline{\mu} \in \Theta$,
$$ \mathbb{E}_{\underline{\mu}}((\delta(\underline{X})|S=s))=\mu_{s},~ \forall \underline{\mu} \in \Theta.$$ \par 
A natural estimator for estimating $\mu_S$, the underlying true response of the selected treatment, is the weighted average of sample means at two stages, i.e. 
\begin{equation}
\delta_M(\underline{X})=\frac{n_1 \overline{X}_S+n_2\overline{Y}}{n_1+n_2}.
\end{equation}
Clearly, $\delta_M(\underline{X})$ is the maximum likelihood estimator (MLE) of $\mu_{S}$. It is worth mentioning here that the statistic $\underline{T}=\left(\frac{n_1\overline{X}_S+n_2\overline{Y}}{n_1+n_2},\overline{X}_{3-S},S\right)=\left(T_1,T_2,S\right)$ (say) is minimal sufficient but not complete. However, given $S=s$, the statistic $(T_1,T_2)$ is a complete-sufficient statistic (cf. \cite{bowden2008unbiased}). Conditioned on $S=s$, \cite{bowden2008unbiased} derived the uniformly minimum variance conditionally unbiased estimator (UMVCUE) of $\mu_{S}$ as
\begin{equation}
\delta_{BG}(\underline{X})=\frac{n_1\overline{X}_S+n_2 \overline{Y}}{n_1+n_2}-\sigma\sqrt{\frac{n_1}{n_2(n_1+n_2)}}\frac{\phi\left(\sqrt{\frac{n_1}{n_2(n_1+n_2)}}\frac{\left(n_2D_2-(n_1+n_2)D_1\right)}{\sigma}\right)}{\Phi\left(\sqrt{\frac{n_1}{n_2(n_1+n_2)}}\frac{\left(n_2D_2-(n_1+n_2)D_1\right)}{\sigma}\right)},
\end{equation} 
where, $D_1=\overline{X}_{3-S}-\overline{X}_S$ and $D_2=\overline{Y}-\overline{X}_{S}$. \par 
In this paper, we aim to study properties of these natural estimators $\delta_{M}$ and $\delta_{BG}$. Specifically we wish to explore if these estimators can be improved. In this direction, in Section 3 of the paper, we prove that the natural estimator $\delta_M$ is minimax and admissible. In Section 4, we prove that the UMVCUE $\delta_{BG}$ is inadmissible for estimating $\mu_{S}$ and we obtain an estimator improving upon it. In the process we derive a sufficient condition for inadmissibility of an arbitrary location and permutation equivariant estimator and obtain dominating estimators in situations where this sufficient condition is satisfied. We then apply this general result to various single stage and two stage estimators and find improved estimators. In section 5, we make numerical comparisons of various competing estimators. A real-life data analysis is presented in Section 6.

The following two lemmas will be useful in proving main results of the paper.

\begin{lemma} \label{S3,L1}
	For any $a \in \mathbb{R}$ and $b \in \mathbb{R}$,
	\begin{itemize}
		
		\item [(i)] $\displaystyle\int_{-\infty}^{\infty} \Phi \left( ax+b\right)\phi(x)~dx= \Phi\left( \frac{b}{\sqrt{1+a^2}}\right)$;
		\item [(ii)] $\displaystyle\int_{-\infty}^{\infty} x\phi \left( ax+b\right)\phi(x)~dx= \frac{-ab}{(1+a^2)^{3/2}}\phi\left( \frac{b}{\sqrt{1+a^2}}\right)$;
		\item [(iii)] $\displaystyle\int_{-\infty}^{\infty} x^2\Phi \left( ax+b\right)\phi(x)~dx=\Phi\left( \frac{b}{\sqrt{1+a^2}}\right)- \frac{a^2 b}{(1+a^2)^{3/2}}\phi\left( \frac{b}{\sqrt{1+a^2}}\right)$.	
	\end{itemize}	
	\begin{proof}
		For  proofs of identities $(i)$ and $(ii)$, refer to Lemma $3.2$ of \cite{masihuddin2021equivariant}. \\
		$(iii)$ Using the relation, $x\phi(x)=-\phi'(x)$ and an application of the identities given in $(i)$ and $(ii)$ yields
		\begin{align*}
		\displaystyle\int_{-\infty}^{\infty} x^2\Phi \left( ax+b\right)\phi(x)~dx&= \displaystyle\int_{-\infty}^{\infty} \left\{\Phi(ax+b)+ax\phi(ax+b)\right\}\phi(x)~dx\\
		&=\Phi\left( \frac{b}{\sqrt{1+a^2}}\right)- \frac{a^2 b}{(1+a^2)^{3/2}}\phi\left( \frac{b}{\sqrt{1+a^2}}\right).
		\end{align*}
	\end{proof}
\end{lemma} 
\begin{lemma} \label{S3,L2}
	Let $U=\frac{n_1\overline{X}_S+n_2\overline{Y}}{n_1+n_2}-\mu_S$,
	$  \sigma^2_{*}= \frac{\sigma^2}{n_1+n_2}$ and $\rho=\frac{n_2}{n_1+n_2}$. Then,
	\begin{itemize}
		\item[(i)]  the p.d.f. of $U$ is given by, $$f_U(u)=\frac{1}{\sigma_{*}}\left[\Phi\left(\frac{\sqrt{n_1}(u+\theta)}{\sigma\sqrt{1+\rho}}\right)+\Phi\left(\frac{\sqrt{n_1}(u-\theta)}{\sigma\sqrt{1+\rho}}\right)\right]\phi\left(\frac{u}{\sigma_{*}}\right),~ -\infty <u < \infty;$$
		\item[(ii)] $\mathbb{E}_{\mu} \left(U^2\right)=\sigma^2_{*}= \frac{\sigma^2}{n_1+n_2}.$
	\end{itemize}
\end{lemma}
\begin{proof} See the Appendix.
\end{proof}
In case of single stage sampling alone, it is well known that the naive estimator $\delta_{0}(\underline{X})=\overline{X}_S$ has good risk properties. It is minimax and admissible under the squared error loss function (see \cite{sackrowitz1986evaluating} and \cite{hwang1993empirical}).  Therefore, it is natural to investigate the minimaxity and admissibility of the two stage natural estimator $\delta_M(\underline{X})$ defined in (2.3). \par 
\par In the following section we prove that the naive estimator $\delta_M(\underline{X})=\frac{n_1\overline{X}_S+n_2\overline{Y}}{n_1+n_2}$ is minimax and admissible for estimating the selected treatment mean $\mu_{S}$ under the squared error loss function (2.2).
\section{Minimaxity and admissibility of the MLE $\delta_M(\underline{X})$}
In the Bayesian set-up, suppose that the unknown parameter $\underline{\mu}=(\mu_1,\mu_2) \in \Theta$ is considered to be a realization of a random vector $\underline{R}=(R_1,R_2)$, having a specified probability distribution called prior distribution. Consider a sequence of prior densities $\left\{\Pi_m\right\}_{m\geq 1}$ for  $\underline{R}=(R_1,R_2)$ as follows:
\begin{equation}
\Pi_m(r_1,r_2)=\frac{1}{2\pi m^2}e^{-\frac{\sum\limits_{i=1}^{2} r^2_i}{2m^2}}, ~-\infty< r_i < \infty,~ i=1,2, ~m=1,2, \ldots
\end{equation}

Recall that,  $\overline{X}_i ~(i=1,2)$ is the sample mean of the first stage sample from the $i^{th}$ population and $\overline{Y}$ is the second stage sample mean of the sample drawn from the population selected at the first stage. \par
It is easy to see that, the posterior distribution of $(R_1,R_2)$ given $(\overline{X}_1, \overline{X}_2,\overline{Y})=(x_1,x_2,y)$ with respect to the prior distribution $\Pi_{m}~(m=1,2,\ldots)$ is such that, the random variables $R_1$ and $R_2$ are independently distributed as $N\left(\lambda_{1,m}, \tau^2_{1,m}\right)$ and $N\left(\lambda_{2,m}, \tau^2_{2,m}\right)$ respectively, where 
\begin{equation*}
\left(\lambda_{1,m}, \lambda_{2,m}, \tau^2_{1,m}, \tau^2_{2,m}\right)=\begin{cases}
\left(\frac{n_1x_1+n_2y}{n_1+n_2+\frac{\sigma^2}{m^2}},\frac{n_1x_2}{n_1+\frac{\sigma^2}{m^2}},\frac{1}{\frac{n_1+n_2}{\sigma^2}+\frac{1}{m^2}},\frac{1}{\frac{n_1}{\sigma^2}+\frac{1}{m^2}}\right), & \text{if $x_1  \geq x_2$ } \vspace{2mm}\\ 
\left(\frac{n_1x_1}{n_1+\frac{\sigma^2}{m^2}},\frac{n_1x_2+n_2y}{n_1+n_2+\frac{\sigma^2}{m^2}}, \frac{1}{\frac{n_1}{\sigma^2}+\frac{1}{m^2}},\frac{1}{\frac{n_1+n_2}{\sigma^2}+\frac{1}{m^2}}\right), & \text{if $x_1 < x_2$}
\end{cases}.
\end{equation*}
Therefore, under the squared error loss function \eqref{S2.E2}, the Bayes estimator of the selected treatment mean $\mu_S$ is
\begin{align}
\delta_{\Pi_m}\left(\underline{X}\right)
&=\begin{cases*}
\mathbb{E}\left(R_1|\underline{X}\right), & \text{if $\overline{X}_1 \geq \overline{X}_2$ }\\
\mathbb{E}\left(R_2|\underline{X}\right), & \text{if $\overline{X}_1 <  \overline{X}_2$}
\end{cases*} \nonumber\\
&=\begin{cases*}
\frac{\left(n_1 \overline{X}_1 + n_2\overline{Y}\right)}{(n_1+n_2)+\frac{\sigma^2}{m^2}}, & \text{if $\overline{X}_1 \geq \overline{X}_2$ }\\
\frac{\left(n_1 \overline{X}_2 + n_2\overline{Y}\right)}{(n_1+n_2)+\frac{\sigma^2}{m^2}}, & \text{if $\overline{X}_1 <  \overline{X}_2$}
\end{cases*} \nonumber\\
&=\frac{\left(n_1\overline{X}_S+n_2\overline{Y}\right)}{(n_1+n_2)+\frac{\sigma^2}{m^2}}.
\end{align}
The posterior risk of the Bayes rule $\delta_{\Pi_m}\left(\underline{X}\right)$ is obtained as
\begin{equation*}
r_{\delta_{\Pi_m}}\left(\underline{X}\right)=\frac{1}{\frac{n_1+n_2}{\sigma^2}+\frac{1}{m^2}},
\end{equation*}
which does not depend on $\underline{X}=(\overline{X}_1,\overline{X}_2,\overline{Y})$. Hence the Bayes risk of the estimator $\delta_{\Pi_m}$ is 
\begin{equation}
r^*_{\delta_{\Pi_m}}\left(\Pi_m\right)=\frac{1}{\frac{n_1+n_2}{\sigma^2}+\frac{1}{m^2}},~ m=1,2,\ldots
\end{equation}
Using Lemma $2.2 ~(ii)$, we have the risk of the estimator $\delta_M(\underline{X})=\frac{n_1\overline{X}_S+n_2\overline{Y}}{n_1+n_2}$, as
\begin{align}
R(\underline{\mu},\delta_{M})&=\mathbb{E}_{\underline{\mu}}\left(\frac{n_1\overline{X}_S+n_2\overline{Y}}{n_1+n_2}-\mu_S\right)^2 \nonumber \\
&=\frac{\sigma^2}{n_1+n_2}.
\end{align}
Therefore, the Bayes risk of the natural estimator $\delta_{M}(\underline{X})$, under the prior $\Pi_m$, is
\begin{align}
r^*_{\delta_{M}}\left(\Pi_m\right)&=\frac{\sigma^2}{n_1+n_2},~ m=1,2,\ldots
\end{align}
Now, we provide the following theorem which proves the minimaxity of the natural estimator $\delta_{M}$.
\begin{theorem}
	The MLE $\delta_M(\underline{X})=\frac{n_1\overline{X}_S+n_2\overline{Y}}{n_1+n_2}$ is minimax for estimating the selected treatment mean $\mu_{S}$ under the squared error loss function \eqref{S2.E2}.
\end{theorem}
\begin{proof}
	Consider $\delta$ to be any other estimator. Since, $\delta_{\Pi_m}$, given by $(3.2)$, is the Bayes estimator of $\mu_{S}$ with respect to the sequence of priors $\Pi_m$, we have,
	\begin{align*}
	\sup_{\underline{\mu} \in \Theta } R(\underline{\mu}, \delta) &\geq \int_{\Theta} R(\underline{\mu},\delta) \Pi_{m}(\mu_{1},\mu_{2})d\underline{\mu}\\
	&\geq \int_{\Theta} R(\underline{\mu},\delta_{\Pi_{m}}) \Pi_{m}(\mu_{1},\mu_{2})d\underline{\mu}\\
	&=r^*_{\delta_{\Pi_{m}}}(\Pi_{m})=\frac{1}{\frac{(n_1+n_2)}{\sigma^2}+\frac{1}{m^2}}, ~~~m=1,2,\ldots ~~~(\text{using}~ (3.3))\\
	\Rightarrow \sup_{\underline{\mu} \in \Theta } R(\underline{\mu}, \delta) &\geq \lim_{m\to\infty}r^*_{\delta_{\Pi_{m}}}(\Pi_{m})=\frac{\sigma^2}{n_1+n_2}=\sup_{\underline{\mu} \in \Theta } R(\underline{\mu}, \delta_{M}), ~~~(\text{using} ~(3.4))
	\end{align*}
	which implies that $\delta_{M}(\underline{X})=\frac{n_1\overline{X}_S+n_2\overline{Y}}{n_1+n_2}$ is minimax for estimating $\mu_{S}$.	
\end{proof}
Now we apply the principle of invariance.
The  estimation problem at hand is invariant under the location group of transformations $\mathcal{G}=\{g_{b}:  b \in {\mathbb{R}}\}$, where $g_{b}(x_{1}, x_{2}, y )=(x_{1}+b, {x}_{2}+b, y+b), (x_{1}, x_{2}, y) \in \mathbb{R}^{3} , b \in {\mathbb{R}}$, and also under the group of permutations $\mathcal{G}_p=\{g_1,g_2\}$, where $g_{1}(x_{1}, x_{2}, y)=(x_1,x_2,y),~ g_{2}(x_{1}, x_{2}, y)=(x_2,x_1,y), (x_{1}, x_{2}, y) \in \mathbb{R}^{3}$.  For the goal of estimating $\mu_{S}$, it is easy to verify that any location and permutation equivariant estimator will be of the form
\begin{equation} \label{S2,E3}
\delta_{\psi}(\underline{X})=\frac{n_1\overline{X}_S+n_2\overline{Y}}{n_1+n_2} + \psi\left(D_1,D_2\right),
\end{equation}
for some function $\psi:(-\infty,0] \times \mathbb{R} \rightarrow \mathbb{R}$, where $D_1=\overline{X}_{3-S}-\overline{X}_S$ and $D_2=\overline{Y}-\overline{X}_S$.  
Let $\mathcal{D}_1$ denote the class of all location and permutation equivariant estimators of the type (3.6). Note that the MLE $\delta_M(\underline{X})$  and UMVCUE $\delta_{BG}$, defined by (2.3) and $(2.4)$  respectively, belong to the class $\mathcal{D}_1.$  

Next, we will show that the MLE $\delta_{M}$ is admissible in the class $\mathcal{D}_1$ of location and permutation equivariant estimators. Notice that, the risk function of any estimator $\delta \in \mathcal{D}_1$ depends on $\underline{\mu} \in \Theta$ through $\theta=\theta_2-\theta_1$. Therefore, we denote $R(\underline{\mu},\delta)$ by $R_{\delta}(\theta)$, $\theta \geq 0.$ Consequently, the Bayes risk of any estimator $\delta \in \mathcal{D}_1$ depends on the prior distribution of $R=(R_1,R_2)$ through the distribution of $(R_{[2]}-R_{[1]})$, where $R_{[1]}=\min\left\{R_1,R_2\right\}$, $R_{[2]}=\max\left\{R_1,R_2\right\}.$ \par 
The following theorem establishes the admissibility of the estimator $\delta_{M}$, within the class $\mathcal{D}_1$ of location and permutation equivariant estimators.
\begin{theorem}
For estimating $\mu_{S}$, under the squared error loss function \eqref{S2.E2}, the MLE  $\delta_{M}(\underline{X})=\frac{n_1\overline{X}_S+n_2\overline{Y}}{n_1+n_2}$ is admissible within the class $\mathcal{D}_1$.
\end{theorem}
\begin{proof}
	On contrary, let us assume that $\delta_{M}$ is inadmissible in $\mathcal{D}_1$, i.e., there exists an estimator $\delta'$ such that
	$$R_{\delta'}(\theta) \leq R_{\delta_{M}}(\theta),~ \forall~ \theta \geq 0$$
	and 
	$$R_{\delta'}(\theta) < R_{\delta_{M}}(\theta),~ \text{for some}~ \theta_0 \geq 0.$$
	
	Let $R_{\delta_{M}}(\theta_{0})-R_{\delta^{'}}(\theta_{0})=\epsilon$, so that $\epsilon>0$. Define $$D=\Big\{\theta \geq0 : R_{\delta_{M}}(\theta_{0})-R_{\delta^{'}}(\theta_{0})>\frac{\epsilon}{2}\Big\}.$$ Then $D \neq \phi~ (\theta_{0}\in D)$ and $D$ is open (since $R_{\delta_{M}}(\theta_{0})-R_{\delta^{'}}(\theta_{0})$ is a continuous function of $\theta \in \left[0,\infty \right))$. Let $r_{\delta^{'}}(\Pi_{m})$ denote the Bayes risk of $\delta'$ under the prior $\Pi_{m}$, $m=1,2\ldots$. Since $\delta_{\Pi_{m}}$, defined by $(3.2)$, is the Bayes estimator  with respect to prior $\Pi_{m}$, we have
	\begin{align}
	r_{\delta_{M}}(\Pi_{m})-r_{\delta_{\Pi_{m}}}(\Pi_{m}) &\geq r_{\delta_{M}}(\Pi_{m})-r_{\delta'}(\Pi_{m}) \nonumber\\
	&= \int_{\Theta}[R_{\delta_{M}}(\theta)-R_{\delta'}(\theta)]\Pi_{m}(\mu_{1},\mu_{2})d\underline{\mu} \nonumber\\
	&\geq \int_{D}[R_{\delta_{M}}(\theta)-R_{\delta^{'}}(\theta)]\Pi_{m}(\mu_{1},\mu_{2})d\underline{\mu}\nonumber\\
	&\geq \frac{\epsilon}{2}~\mathbb{P}_{\Pi_{m}}\left(\left(R_{[2]}-R_{[1]}\right)\in D \right)  , ~~~m=1,2,\ldots.
	\end{align}
	Now using $(3.3), (3.5)$ and $(3.7)$, we get
	
	\begin{align}
	~~~~~~~\liminf\limits_{m \rightarrow \infty} \frac{1}{ m^2 \left(\frac{n_1+n_2}{\sigma^2}\right) \left(\frac{n_1+n_2}{\sigma^2}+\frac{1}{m^2}\right)\mathbb{P}_{\Pi_{m}}\left(\left(R_{[2]}-R_{[1]}\right) \in D\right)} & \geq \frac{\epsilon}{2}. 
	\end{align}	
	However, for any $0 \leq a < b < \infty$, we have
	\begin{align*}
	\mathbb{P}_{\Pi_{m}}\left(a <R_{[2]}-R_{[1]} \leq b\right)&=2\left[\Phi\left(\frac{b}{\sqrt{2}m}\right)-\Phi\left(\frac{a}{\sqrt{2}m}\right)\right]\\
	&=\frac{\sqrt{2}(b-a)}{m}\phi(d^*),~~\frac{a}{\sqrt{2}m}<d^*<\frac{b}{\sqrt{2}m}.
	\end{align*}
	Therefore,

	\begin{align}
	&\liminf\limits_{m \rightarrow \infty} \frac{1}{ m^2 \left(\frac{n_1+n_2}{\sigma^2}\right) \left(\frac{n_1+n_2}{\sigma^2}+\frac{1}{m^2}\right)\mathbb{P}_{\Pi_{m}}\left(\left(R_{[2]}-R_{[1]}\right) \in D\right)}\\
	&=\liminf\limits_{m \rightarrow \infty} \frac{1}{ m^2 \left(\frac{n_1+n_2}{\sigma^2}\right) \left(\frac{n_1+n_2}{\sigma^2}+\frac{1}{m^2}\right)\frac{\sqrt{2}(b-a)}{m}\phi(d^*)}=0,
	\end{align}	
	which contradicts $(3.8)$. Hence, the theorem follows.
\end{proof}

\begin{remark}
	Note that $\delta_{0}(\underline{X})=\overline{X}_S$ is an admissible and minimax estimator based on the single stage data (see \cite{sackrowitz1986evaluating}). Using Lemma $3.1~(ii)$ of \cite{masihuddin2021equivariant} and Lemma $2.2~(ii)$ we have
	$$R\left(\underline{\mu},\delta_{M}\right)=\frac{\sigma^2}{n_1+n_2}< \frac{\sigma^2}{n_1}=R\left(\underline{\mu},\delta_{0}\right),~ \forall~ \underline{\mu} \in \Theta,$$
	i.e., under the two stage drop-the-losers design, the estimator $\delta_{0}\left(\underline{X}\right)=\overline{X}_S$ is dominated by the two- stage admissible and minimax estimator $\delta_{M}$.
\end{remark}
 \par


\section{Inadmissibility of the UMVCUE}
It is a well known result that for Gaussian populations there does not exists an unbiased estimator of the selected treatment mean $\mu_{S}$ using single stage data only (see \cite{putter1968technical},  \cite{vellaisamy2009note} and \cite{masihuddin2021equivariant}). A similar result was provided by \cite{tappin1992unbiased} for binomial populations. For the goal of estimating $\mu_{S}$, conditioned on $S$, \cite{bowden2008unbiased} proposed the following two stage UMVCUE 
\begin{equation}
\delta_{BG}(\underline{X})=\frac{n_1\overline{X}_S+n_2 \overline{Y}}{n_1+n_2}+\psi_{BG}\left(D_1,D_2\right)
\end{equation}
with, $\psi_{BG}\left(D_1,D_2\right)=-\sigma\sqrt{\frac{n_1}{n_2(n_1+n_2)}}\frac{\phi\left(\sqrt{\frac{n_1}{n_2(n_1+n_2)}}\frac{\left(n_2D_2-(n_1+n_2)D_1\right)}{\sigma}\right)}{\Phi\left(\sqrt{\frac{n_1}{n_2(n_1+n_2)}}\frac{\left(n_2D_2-(n_1+n_2)D_1\right)}{\sigma}\right)}$,  $D_1=\overline{X}_{3-S}-\overline{X}_S$ and $D_2=\overline{Y}-\overline{X}_{S}$. It is worth mentioning here that, conditioned on $S$, $\overline{Y}$ is an unbiased estimator of $\mu_{S}$, $\left(T_1,T_2\right)=\left(\frac{n_1\overline{X}_S+n_2\overline{Y}}{n_1+n_2},\overline{X}_{3-S}\right)$ is a complete-sufficient for $\underline{\mu}$ and $\mathbb{E}_{\underline{\mu}}\left(\overline{Y}|(T_1,T_2)\right)=\delta_{BG}(\underline{X})$.
Clearly $\delta_{BG}(\underline{X}) \in \mathcal{D}_1$, where $\mathcal{D}_1$ is the class of location and permutation equivariant estimators defined in the last section (see (3.6)). In order to find an improvement over $\delta_{BG}$, we now use the idea of \cite{brewster1974improving} to obtain a sufficient condition for inadmissibility of an arbitrary location and permutation equivariant estimators $\delta_{\psi}(\underline{X}) = \frac{n_1\overline{X}_S+n_2\overline{Y}}{n_1+n_2} +\psi(D_1,D_2) $, for some function function $\psi(\cdot):\left(-\infty,0 \right] \times 
\mathbb{R} \to \mathbb{R}$, of $\mu_{S}$. As a direct consequence of this result, we will see later in the section that the estimator $\delta_{BG}(\cdot)$, given in $(4.1)$, is inadmissible and a dominating estimator is provided.\par
Recall that $\overline{X}_{3-S} = \min(\overline{X}_1 ,\overline{X}_2), \overline{X}_S = \max (\overline{X}_1 ,\overline{X}_2)$; $D_1 = \overline{X}_{3-S} -\overline{X}_{S}$;  $D_2 =  \overline{Y} - \overline{X}_{S}$. Let $S_1=\overline{X}_{S}-\mu_S$. 
The following lemma will be useful in obtaining the main result of this section.
\begin{lemma} \label{S4,L2}
	Let $d_1 \in (-\infty,0]$ and $d_2 \in \mathbb{R}$ be fixed. Then, 	
\begin{itemize}
\item[(a)] the conditional p.d.f. of $S_1$, given $\left(D_1, D_2\right)=(d_1,d_2)$, is given by
	\begin{align*}
		f_{1, \underline{\mu}}(s|d_1,d_2)&= \frac{\left[\phi\left(\frac{\sqrt{n_1}}{\sigma}(s+d_1-\theta)\right)+\phi\left(\frac{\sqrt{n_1}}{\sigma}(s+d_1+\theta)\right)\right]\phi\left(\frac{\sqrt{n_2}}{\sigma}(s+d_2)\right)\phi\left(\frac{\sqrt{n_1}}{\sigma}s\right)}{\displaystyle{\int_{-\infty}^{\infty}}\left[\phi\left(\frac{\sqrt{n_1}}{\sigma}(t+d_1-\theta)\right)+\phi\left(\frac{\sqrt{n_1}}{\sigma}(t+d_1+\theta)\right)\right]\phi\left(\frac{\sqrt{n_2}}{\sigma}(t+d_2)\right)\phi\left(\frac{\sqrt{n_1}}{\sigma}t\right) dt } ~,s\in \mathbb{R};
	\end{align*}	
\item[(b)]the conditional expectation of $S_1$, given $\left(D_1, D_2\right)=(d_1,d_2)$, is given by
$$\mathbb{E}_{\underline{\mu}}\left[S_1|(D_1,D_2)=(d_1,d_2)\right]=\frac{n_1 \theta}{n_2+2n_1} \left\{\frac{1-e^{-\frac{2 \theta}{\sigma^2}\left(\frac{n_1((n_1+n_2)d_1-n_2d_2)}{2n_1+n_2}\right)}}{1+e^{-\frac{2 \theta}{\sigma^2}\left(\frac{n_1((n_1+n_2)d_1-n_2d_2)}{2n_1+n_2}\right)}}\right\}-\frac{n_1d_1+n_2d_2}{2n_1+n_2}.$$

	\end{itemize}
		\end{lemma}
	\begin{proof} See the Appendix.
	\end{proof}
It is worth noting that the risk function of the estimator $\delta_{\psi}(\underline{X})$, given by $(3.6)$, depends on $\underline{\mu}=(\mu_1,\mu_2)$ only through $\theta$. Therefore, for notational convenience, we denote $R(\underline{\mu},\delta_{\psi})$ by $R_{\theta}(\delta_{\psi})$. The risk function of an estimator of the form $(3.6)$ is 
\begin{align}
R_{\theta}(\delta_{\psi})&=\mathbb{E}_{\underline{\mu}}\left[R_1\left(\theta,\psi\left(D_1,D_2\right)\right)\right],~ \theta \geq 0,
\end{align}
where, for any fixed $d_1 \in (-\infty,0]$ and $d_2 \in \mathbb{R}$,
\begin{align}
R_1\left(\theta,\psi\left(d_1,d_2\right)\right)&=\mathbb{E}_{\underline{\mu}} \left[\left(\frac{n_1\overline{X}_S+n_2\overline{Y}}{n_1+n_2}+\psi(D_1,D_2)-\mu_S\right)^2\Bigg| (D_1,D_2)=(d_1,d_2)\right], ~ \theta \geq 0.
\end{align}
Our objective is to minimize the risk function $R_{\theta}(\delta_{\psi})$ given by $(4.2)$ with respect to $\psi$. This objective may be achieved by minimizing the conditional risk $R_1\left(\theta,\psi\left(d_1,d_2\right)\right)$, given by $(4.3)$. For any fixed $\theta \in [0,\infty)$, the choice of $\psi$ that minimizes $(4.3)$ is obtained as
\begin{align}
\psi_{\theta}(d_1,d_2)&=-\mathbb{E}_{\underline{\mu}} \left[\left(\frac{n_1\overline{X}_S+n_2\overline{Y}}{n_1+n_2}-\mu_S\right)\Big| (D_1,D_2)=(d_1,d_2)\right] \nonumber \\
&=-\mathbb{E}_{\underline{\mu}}\left[S_1|(D_1,D_2)=(d_1,d_2)\right]-\frac{n_2d_2}{n_1+n_2}
\end{align}
 Further, an application of Lemma $4.1(b)$ implies,
\begin{align*}
\psi_{\theta}(d_1,d_2)&= \frac{n_1 \theta}{n_2+2n_1} \left\{\frac{1-e^{-\frac{2 \theta}{\sigma^2}\left(\frac{n_1((n_1+n_2)d_1-n_2d_2)}{2n_1+n_2}\right)}}{1+e^{-\frac{2 \theta}{\sigma^2}\left(\frac{n_1((n_1+n_2)d_1-n_2d_2)}{2n_1+n_2}\right)}}\right\}+\frac{n_1((n_1+n_2)d_1-n_2d_2)}{(n_1+n_2)(2n_1+n_2)},~\theta \geq0.
\end{align*}
It can be seen that $\psi_{\theta}(d_1,d_2)$ is increasing in $\theta \in [0,\infty)$, whenever $d_1 \leq \frac{n_2 d_2}{n_1+n_2}$, and decreasing in $\theta \in [0,\infty)$, whenever $d_1 > \frac{n_2 d_2}{n_1+n_2}.$ 
Thus, we have
\begin{equation}
\inf\limits_{\theta \geq 0} \psi_{\theta}(d_1,d_2)= \begin{cases}
\frac{n_1((n_1+n_2)d_1-n_2d_2)}{(n_1+n_2)(2n_1+n_2)}, & \text{if $d_1 \leq \frac{n_2 d_2}{n_1+n_2}$ }\\
-\infty, & \text{if $d_1 > \frac{n_2 d_2}{n_1+n_2} $}
\end{cases}=\psi_{*}(d_1,d_2),~ (\text{say})
\end{equation}
and 
\begin{equation}
\sup\limits_{\theta \geq 0} \psi_{\theta}(d_1,d_2)= \begin{cases}
\infty, & \text{if $d_1 \leq \frac{n_2 d_2}{n_1+n_2}$ }\\
\frac{n_1((n_1+n_2)d_1-n_2d_2)}{(n_1+n_2)(2n_1+n_2)}, & \text{if $d_1 > \frac{n_2 d_2}{n_1+n_2} $}
\end{cases}=\psi^{*}(d_1,d_2),~ (\text{say}).
\end{equation}
	Note that, for any fixed $d_1 \in (-\infty,0]$, $d_2 \in \mathbb{R}$ and $\underline{\mu} \in \Theta$, the conditional risk, given by $(4.3)$,
is strictly decreasing on $ \left(-\infty, \psi_{\theta}(d_1,d_2) \right)$ and strictly increasing on $ \left(\psi_{\theta}(d_1,d_2), \infty \right)$, where $ \psi_{\theta}(d_1,d_2),~ \theta \geq 0,$ is defined by $(4.4)$. Using this fact along with $(4.5)$ and $(4.6)$, we also note that:\\ $(i)$ for any fixed $d_1 \in (-\infty,0]$, $d_2 \in \mathbb{R}$ and $\underline{\mu} \in \Theta$, $R_1\left(\theta,\psi\left(d_1,d_2\right)\right)$ is strictly decreasing in $\psi\left(d_1,d_2\right)  \in (-\infty,\psi_{*}(d_1, d_2))$;\\
 $(ii)$ for any fixed $d_1 \in (-\infty,0]$, $d_2 \in \mathbb{R}$ and  $\underline{\mu} \in \Theta$, $R_1\left(\theta,\psi\left(d_1,d_2\right)\right)$ is strictly increasing in $\psi\left(d_1,d_2\right)  \in \left(\psi^{*}(d_1, d_2), \infty \right)$. \par 
Thus, we have the following theorem.

\begin{theorem}
	 Let  $\delta_{\psi}\left(\underline{X}\right)=\frac{n_1\overline{X}_S+n_2\overline{Y}}{n_1+n_2}+\psi\left(D_1,D_2\right) $, for a real valued function $\psi $ defined on $:(-\infty,0] \times \mathbb{R}$, be a location and permutation equivariant estimator of $\mu_{S}$. Suppose that
	$$\mathbb{P}_{\underline{\mu}}\left[\psi(D_1, D_2)< \frac{n_1((n_1+n_2)D_1-n_2D_2)}{(n_1+n_2)(2n_1+n_2)} \leq 0 \right] + \mathbb{P}_{\underline{\mu}}\left[ 0 \leq \frac{n_1((n_1+n_2)D_1-n_2D_2)}{(n_1+n_2)(2n_1+n_2)} \leq \psi(D_1, D_2) \right]  >0,~ \text{for some}~ \underline{\mu} \in \Theta, $$
	where $\psi_*(D_1, D_2) $ and $\psi^*(D_1,D_2) $ are defined in $(4.5)$ and $(4.6)$ respectively.
	Then, the estimator $\delta_{\psi}(\cdot)$ is inadmissible for estimating $\mu_{S}$ under the squared error loss function \eqref{S2.E2} and is dominated by
	\begin{equation}
	\delta_{\psi}^I\left(\underline{X}\right)=\begin{cases}
	\frac{n_1\left(\overline{X}_1+\overline{X}_2\right)+n_2\overline{Y}}{2n_1+n_2}, & \text{ if }~ \psi(D_1, D_2) < \frac{n_1}{2n_1+n_2}\left(\overline{X}_{3-S}-\frac{n_1\overline{X}_{S}+n_2\overline{Y}}{n_1+n_2}\right) \leq 0~ ,\\
	 & \text{ or }~ 0 \leq \frac{n_1}{2n_1+n_2}\left(\overline{X}_{3-S}-\frac{n_1\overline{X}_{S}+n_2\overline{Y}}{n_1+n_2}\right) < \psi(D_1, D_2),\\
		\delta_{\psi}\left(\underline{X}\right), & \text{otherwise}.
	\end{cases}
	\end{equation}

As an immediate consequence of the above theorem we have the following corollary dealing with the admissibility of the UMVCUE $\delta_{BG}(\underline{X})$ and $\delta_{0}(\underline{X})=\overline{X}_S$.
\begin{corollary}
	\begin{itemize}
	\item [(a)] 
	The UMVCUE ~$\delta_{BG}(\underline{X})$, given in $(4.1)$, is inadmissible for estimation $\mu_{S}$ under the squared error loss function \eqref{S2.E2} and is dominated by
	\begin{align}
	\delta^I_{BG}(\underline{X})&=\begin{cases}
	\frac{n_1(\overline{X}_1+\overline{X}_2)+n_2 \overline{Y}}{2n_1+n_2}, & \text{ if }~ \overline{X}_{3-S} \leq \frac{n_1\overline{X}_S+n_2\overline{Y}}{n_1+n_2} \leq \overline{X}_{3-S} + \frac{\left(2n_1+n_2\right)\sigma}{\sqrt{n_1 n_2(n_1+n_2)}}\frac{\phi\left(Q\right)}{\Phi\left(Q\right)} \\
	\delta_{BG}(\underline{X}), & \text{otherwise}.
	\end{cases}
	\end{align}
	where, $Q=\sqrt{\frac{n_1}{n_2(n_1+n_2)}}\frac{\left(n_1\overline{X}_S+n_2\overline{Y}-(n_1+n_2)\overline{X}_{3-S}\right)}{\sigma}$.
	\item [(b)] The single stage admissible and minimax estimator $\delta_{0}(\underline{X})=\overline{X}_S=\frac{n_1\overline{X}_S+n_2 \overline{Y}}{n_1+n_2}-\frac{n_2D_2}{n_1+n_2}$  is also inadmissible for estimating $\mu_{S}$ and is dominated by 
	\begin{align}
	\delta^I_{0}(\underline{X})&=\begin{cases}
		\frac{n_1(\overline{X}_1+\overline{X}_2)+n_2 \overline{Y}}{2n_1+n_2}, & \text{ if }~ -\frac{n_2(\overline{Y}-\overline{X}_S)}{n_1+n_2} <\frac{n_1}{2n_1+n_2}\left(\overline{X}_{3-S}-\frac{n_1\overline{X}_S+n_2\overline{Y}}{n_1+n_2}\right) \leq 0 \\
	& \text{or},~ 0 \leq \frac{n_1}{2n_1+n_2}\left(\overline{X}_{3-S}-\frac{n_1\overline{X}_S+n_2\overline{Y}}{n_1+n_2}\right)<-\frac{n_2(\overline{Y}-\overline{X}_S)}{n_1+n_2},\\
	\delta_{0}(\underline{X}), & \text{otherwise}.
	\end{cases} 
	\end{align}
\end{itemize}
\end{corollary}
\end{theorem}
\begin{remark}
Under the single stage sample set-up, \cite{dahiya1974estimation} considered various estimators for estimating the selected treatment mean $\mu_{S}$. It follows from Theorem 4.1 that, under the two stage DLD, all the estimators proposed by \cite{dahiya1974estimation} are inadmissible for estimating $\mu_{S}$ w.r.t. the criterion of mean squared error.
\end{remark}
\subsection{Some additional admissibility results}
Note that the random estimand $\mu_{S}$ depends on observations only through the sufficient statistic $\underline{T}=\left(\frac{n_1\overline{X}_S+n_2\overline{Y}}{n_1+n_2},\overline{X}_{3-S},S\right)$\\$=\left(T_1,T_2,S\right)$. Then, for the goal of estimating $\mu_{S}$, an improved estimator $\delta^{I}_{\psi}$ obtained using Theorem 4.1, can be further improved in terms of MSE if it is not purely a function of the sufficient statistic $\underline{T}=\left(\frac{n_1\overline{X}_S+n_2\overline{Y}}{n_1+n_2},\overline{X}_{3-S},S\right)=\left(T_1,T_2,S\right)$. In fact any location and permutation equivariant estimator $\delta_{\psi}$ of $\mu_{S}$ that is not a function of the sufficient statistic can be improved using the Rao-Blackwell theorem on $\delta_{\psi}$ to obtain a dominating estimator as $\delta^*_{\psi}(\underline{T})=\mathbb{E}_{\underline{\mu}}\left(\delta_{\psi}(\underline{X})\Big|\underline{T}\right)$.
For estimating $\mu_{S}$ under the squared error loss \eqref{S2.E2}, suppose we have an estimator of the type
\begin{align*}
\delta_{\psi}(\underline{X})&=\frac{n_1\overline{X}_S+n_2 \overline{Y}}{n_1+n_2} +\psi\left(D_1,D_2\right) \\
&=\frac{n_1\overline{X}_S+n_2 \overline{Y}}{n_1+n_2} +\psi\left(\overline{X}_{3-S}-\overline{X}_S,\overline{Y}-\overline{X}_S\right).
\end{align*} 
Then the estimator 
\begin{align*}
\delta^*_{\psi}(\underline{X})&=\mathbb{E}_{\underline{\mu}}\left[\frac{n_1\overline{X}_S+n_2 \overline{Y}}{n_1+n_2} +\psi\left(\overline{X}_{3-S}-\overline{X}_S,\overline{Y}-\overline{X}_S\right)\Big|\underline{T}\right] \\
&=T_1+\mathbb{E}_{\underline{\mu}}\left[\psi\left(\overline{X}_{3-S}-\overline{X}_S,\overline{Y}-\overline{X}_S\right)\Big|\underline{T}\right],
\end{align*}
dominates $\delta_{\psi}$, i.e. $R(\mu_{S},\delta^*_{\psi}) < R(\mu_{S},\delta_{\psi})$ $\forall~ \underline{\mu} \in \Theta$ (unless $\delta_{\psi}=\delta^*_{\psi}$ a.e.). \par
\textbf{Example}: 
For estimating the selected treatment mean $\mu_{S}$ under the two-stage drop-the-loser design, the single stage admissible and minimax estimator $\delta_{0}(\underline{X})=\overline{X}_S$ can also be improved using the Rao- Blackwell theorem. In terms of MSE, $\delta_{0}$ is dominated by an estimator $\delta^{RB}_{0}(T_1,T_2,S)=\mathbb{E}\left(\delta_{0}(\underline{X})|\underline{T}\right)$ given as
\begin{align}
\delta^{RB}_{0}(\underline{T})&=T_1+ \frac{\sqrt{n_2}\sigma}{\sqrt{n_1(n_1+n_2)}}\frac{\phi\left(\frac{\sqrt{n_1(n_1+n_2)}}{\sqrt{n_2}\sigma}(T_1-T_2)\right)}{\Phi\left(\frac{\sqrt{n_1(n_1+n_2)}}{\sqrt{n_2}\sigma}(T_1-T_2)\right)}.
\end{align}
\textbf{Derivation of $\delta^{RB}_{0}$}:
Let $\overline{Y}_1$, $\overline{Y}_2$ be as defined in the proof of Lemma 2.2, given in the Appendix.	Note that, for $(t_1,t_2) \in \mathbb{R}^2$, $\delta^{RB}_{0}(t_1,t_2,1)=\mathbb{E}_{\underline{\mu}}\left(\overline{X}_1\Big|\frac{n_1\overline{X}_1+n_2\overline{Y}_1}{n_1+n_2}=t_1,\overline{X}_2=t_2,\overline{X}_1>t_2\right)$ and\\ $\delta^{RB}_{0}(t_1,t_2,2)=\mathbb{E}_{\underline{\mu}}\left(\overline{X}_2\Big|\frac{n_1\overline{X}_2+n_2\overline{Y}_2}{n_1+n_2}=t_1,\overline{X}_1=t_2,\overline{X}_2>t_2\right)$. We have, 
	$$ \begin{bmatrix}\overline{X}_1 \\ \frac{n_1\overline{X}_1+n_2\overline{Y}_1}{n_1+n_2} \\ \overline{X}_2 \end{bmatrix} \sim N_3 \left(\begin{bmatrix}\mu_1 \\ \mu_1 \\ \mu_{2} \end{bmatrix} , \begin{bmatrix}\frac{\sigma^2}{n_1} & \frac{\sigma^2}{n_1+n_2}&0 \vspace{2mm}\\ \frac{\sigma^2}{n_1+n_2} & \frac{\sigma^2}{n_1+n_2}&0\\
	0&0&\frac{\sigma^2}{n_1} \end{bmatrix} \right).$$
	
Given that $ \left( \frac{n_1\overline{X}_1+n_2\overline{Y}_1}{n_1+n_2},\overline{X}_2\right)=(t_1,t_2)$, $\overline{X}_1$ follows univariate Gaussian distribution with mean $\mu^*=t_1$ and variance $\sigma^2_1=\frac{n_2\sigma^2}{n_1(n_1+n_2)}$. Therefore,
\begin{align*}
\delta^{RB}_{0}(t_1,t_2,1)&=\mathbb{E}_{\underline{\mu}}\left(\overline{X}_1\Bigg|\frac{n_1\overline{X}_1+n_2\overline{Y}_1}{n_1+n_2}=t_1,\overline{X}_2=t_2,\overline{X}_1>t_2\right)\\&=\frac{\displaystyle{\int_{t_2}^{\infty}}\frac{x}{\sigma_1}\phi\left(\frac{x-t_1}{\sigma_1}\right)dx}{\displaystyle{\int_{t_2}^{\infty}}\frac{1}{\sigma_1}\phi\left(\frac{x-t_1}{\sigma_1}\right)dx}\\
&=\frac{\displaystyle{\int_{\frac{t_2-t_1}{\sigma_1}}^{\infty}}(t_1+\sigma_1y)\phi\left(y\right)dy}{\displaystyle{\int_{\frac{t_2-t_1}{\sigma_1}}^{\infty}}\phi\left(y\right)dy}=t_1+\sigma_1\frac{\phi\left(\frac{t_1-t_2}{\sigma_1}\right)}{\Phi\left(\frac{t_1-t_2}{\sigma_1}\right)}.
\end{align*}
By symmetry,
\begin{align*}
\delta^{RB}_{0}(t_1,t_2,2)=\mathbb{E}_{\underline{\mu}}\left(\overline{X}_2\Bigg|\frac{n_1\overline{X}_2+n_2\overline{Y}_2}{n_1+n_2}=t_1,\overline{X}_1=t_2,\overline{X}_2>t_2\right)&=t_1+\sigma_1\frac{\phi\left(\frac{t_1-t_2}{\sigma_1}\right)}{\Phi\left(\frac{t_1-t_2}{\sigma_1}\right)}.
\end{align*}
Therefore, 
$\delta^{RB}_{0}(\underline{T})=T_1+ \frac{\sqrt{n_2}\sigma}{\sqrt{n_1(n_1+n_2)}}\frac{\phi\left(\frac{\sqrt{n_1(n_1+n_2)}}{\sqrt{n_2}\sigma}(T_1-T_2)\right)}{\Phi\left(\frac{\sqrt{n_1(n_1+n_2)}}{\sqrt{n_2}\sigma}(T_1-T_2)\right)}.$\\
Note that the estimator $\delta^I_{0}$, defined by (4.9), also dominates $\delta_{0}(\underline{X})=\overline{X}_S$.
Since the estimator $\delta^I_{0}(\underline{X})$   is not a function of minimal sufficient statistic $\underline{T}$, it can be further improved by an estimator $\delta_{1}$, defined as
\begin{align*}
&\delta_{1}(\underline{T})\nonumber =\mathbb{E}\left(\delta^I_{0}(\underline{X})|\underline{T}\right)\nonumber\\
&=\begin{cases}
\frac{(n_1+n_2)T_1+n_1 T_2}{2n_1+n_2}\mathbb{P}_{\underline{\mu}}\left( -\frac{n_2(\overline{Y}-\overline{X}_S)}{n_1+n_2} <\frac{n_1\left(T_2-T_1\right)}{2n_1+n_2} \Big|\underline{T}\right)+\mathbb{E}_{\underline{\mu}}\left[\overline{X}_S I\left(-\frac{n_2(\overline{Y}-\overline{X}_S)}{n_1+n_2} \geq \frac{n_1\left(T_2-T_1\right)}{2n_1+n_2}\right) \Big|\underline{T}\right],
& \text{ if }~ T_1 > T_2,\vspace{3mm}\\
\frac{(n_1+n_2)T_1+n_1 T_2}{2n_1+n_2}\mathbb{P}_{\underline{\mu}}\left( -\frac{n_2(\overline{Y}-\overline{X}_S)}{n_1+n_2} \geq \frac{n_1\left(T_2-T_1\right)}{2n_1+n_2} \Big|\underline{T}\right)+\mathbb{E}_{\underline{\mu}}\left[\overline{X}_S I\left(-\frac{n_2(\overline{Y}-\overline{X}_S)}{n_1+n_2} < \frac{n_1\left(T_2-T_1\right)}{2n_1+n_2}\right) \Big|\underline{T}\right], &  \text{ if }~ T_1 \leq T_2.
\end{cases} 
\end{align*} 
Using the conditional distribution of $\overline{X}_S$ given $\left(T_1, T_2\right)=\left(t_1, t_2\right)$ obtained in the derivation of the estimator $\delta^{EB}_{0}$ (4.10) we obtain,
\begin{align}
\delta_{1}(\underline{T})&=\begin{cases}
\frac{(n_1+n_2)T_1+n_1 T_2}{2n_1+n_2}\left\{\frac{\Phi\left(\frac{T_1-T_2}{\sigma_1}\right)-\Phi\left(\frac{n_1(T_1-T_2)}{(2n_1+n_2)\sigma_1}\right)}{\Phi\left(\frac{T_1-T_2}{\sigma_1}\right)}\right\}+\frac{\sigma_1\phi\left(\frac{n_1(T_1-T_2)}{(2n_1+n_2)\sigma_1}\right)+T_1\Phi\left(\frac{n_1(T_1-T_2)}{(2n_1+n_2)\sigma_1}\right)}{\Phi\left(\frac{T_1-T_2}{\sigma_1}\right)}, & \text{ if }~ T_1 > T_2,\vspace{3mm}\\
\frac{(n_1+n_2)T_1+n_1 T_2}{2n_1+n_2} , & \text{ if }~ T_1 \leq  T_2.
\end{cases} ,
\end{align}
where, $\sigma_1=\sigma\sqrt{\frac{n_2}{n_1(n_1+n_2)}}$. 
In addition to numerical comparison of performances of various estimators, in the following section, we compare the performance of the estimator $\delta_{1}(\underline{T})$ with $\delta^I_{0}$ and $\delta^{RB}_{0}$ numerically, through simulation, under the criterion of MSE and Bias. In confirmity with our theoretical findings, numerical study suggests that the estimator $\delta_{1}$, defined by (4.11), has better MSE performance than $\delta^{I}_{0}$, given by (4.9). In terms of MSE, we will also see in the next Section that $\delta_{1}$ performs better than $\delta^{RB}_{0}$, given by (4.10).
 \par 
\section{Simulation Study}
In this section, we consider a simulation study to compare the risk (MSE)  and the bias performances of various estimators of the selected treatment mean ($\mu_{S}$) under the two-stage drop-the-loser design. We consider the following estimators for our numerical study: $\delta_{M}$, $\delta_{BG}$, $\delta^I_{BG}$, $\delta_{0}$, $\delta^I_{0}$, $\delta^{RB}_{0}$ and $\delta_{1}$ (see (2.3), (2.4), (4.8), (4.9), (4.10) and (4.11)). For our simulation study we take $\sigma=1$. Since the MSE and bias of these estimators of $\mu_{S}$, depends on $\underline{\mu}=(\mu_{1},\mu_{2}) \in \Theta$ only through $\theta=\theta_2-\theta_1$ and the first and second stage sample sizes $n_1$ and $n_2$ respectively, we simulate the MSE and the Bias of different estimators against different configurations of $\theta \geq 0$ and sample sizes $n_1$ and $n_2$. For $n_1=5,10,15$ and $n_2=5,10,15$, the simulated values of the MSE and the Bias based on 10,000 simulations are plotted in Figures 5.1-5.12. The percentage risk improvement of the estimator $\delta^I_{0}$ over $\delta_{0}$ have been tabulated in Table 1. It is observed that $\delta^I_{0}$ has considerable risk improvement over the single stage admissible and minimax  estimator $\delta_{0}$. 
 In Figures 5.1-5.6, we have plotted the simulated MSE values of the estimators $\delta_{M}$, $\delta_{BG}$, $\delta^I_{BG}$, $\delta^{RB}_{0}$ and $\delta_{1}$ for different configurations of $n_1, n_2$ against $\theta$. In Figures 5.7-5.12, we have plotted the simulated Bias values of the various estimators, $\delta_{M}$, $\delta_{BG}$, $\delta^I_{BG}$, $\delta^{RB}_{0}$ and $\delta_{1}$ for different configurations of $n_1, n_2$ against $\theta$. \par 

The Following observations are evident from Figures 5.1-5.12:
\begin{itemize}
	\item[(i)] The MSE of the estimator $\delta^I_{BG}$ is nowhere larger than $\delta_{BG}$ (see Figure 5.1-5.6), which confirms our theoretical finding. 
	
	\item[(ii)] In terms of MSE, the maximum likelihood estimator $\delta_{M}$, which is minimax and admissible, performs better than all the estimators of $\mu_{S}$.  For $n_1 \geq n_2 $, the estimator  $\delta_{1}$ beats all the other estimators, except $\delta_M$, in terms of MSE. It is also observed that, $\delta_{1}$ uniformly dominates $\delta^{RB}_{0}$ in terms of MSE.
	
	\item[(iii)] In terms of Bias, the two stage UMVCUE ($\delta_{BG}$) performs better (i.e. it has zero bias) among all the estimators of $\mu_{S}$. After $\delta_{BG}$, $\delta^I_{BG}$ performs better than all the other estimators (see Figure 5.7-5.12).
	\item[(iv)] When both the MSE and the Bias criterion are used to choose an estimator for estimating $\mu_{S}$, $\delta^I_{BG}$ seems to be a good choice.
  \item[(v)] The MSE and Bias of all the competing estimators of $\mu_{S}$ approaches to zero for large values of the sample sizes $n_1$ and $n_2$.
	
\end{itemize}

\begin{table}[h!]
		\textbf{Table 1: Percentage risk improvement of estimator $\delta^I_{0}$ over $\delta_{0}$.}
	\centering

	\begin{tabular}{lcccc}
		\hline
		& \multicolumn{4}{c}{\textbf{Percentage Risk   Improvement}} \\ \cline{2-5} 
		$\theta$ &	$(n_1,n_2)=(5,5)$ &$(n_1,n_2)=(10,5)$ & $(n_1,n_2)=(10,10)$&$(n_1,n_2)=(10,15)$        \\ \hline
		0.00 & 13.83       & 7.10          & 15.01       & 22.17      \\
		0.05 & 15.17       & 6.78          & 15.41      & 21.80        \\
		0.1 & 15.04        &6.92          & 14.80        &21.75     \\
		0.2 & 14.95        & 7.42        & 16.40        &23.89       \\
		0.3 & 17.20        &8.18         & 17.09      & 24.50     \\ 
		0.5 & 17.62        & 8.30         & 17.92       &26.71       \\ 
		1.0 & 16.44        & 4.47        & 10.55         & 18.63     \\ 
		1.5 & 10.16        & 0.69        & 3.33        &  6.37         \\ 
		2.0 & 3.89         &0.063        & 0.51         &  2.05       \\ 
		2.5& 1.67          & 0.00         & 0.049        & 0.44       \\ 
		3.0 & 0.42         & 0.00         & 0.00        & 0.05   \\ 
		\hline
	\end{tabular}
\end{table}
\newpage

\newpage
 \FloatBarrier
\vskip -0.1in
\begin{figure}[!h]
	\centering
	\includegraphics[height=3.6in, width=4.6in]{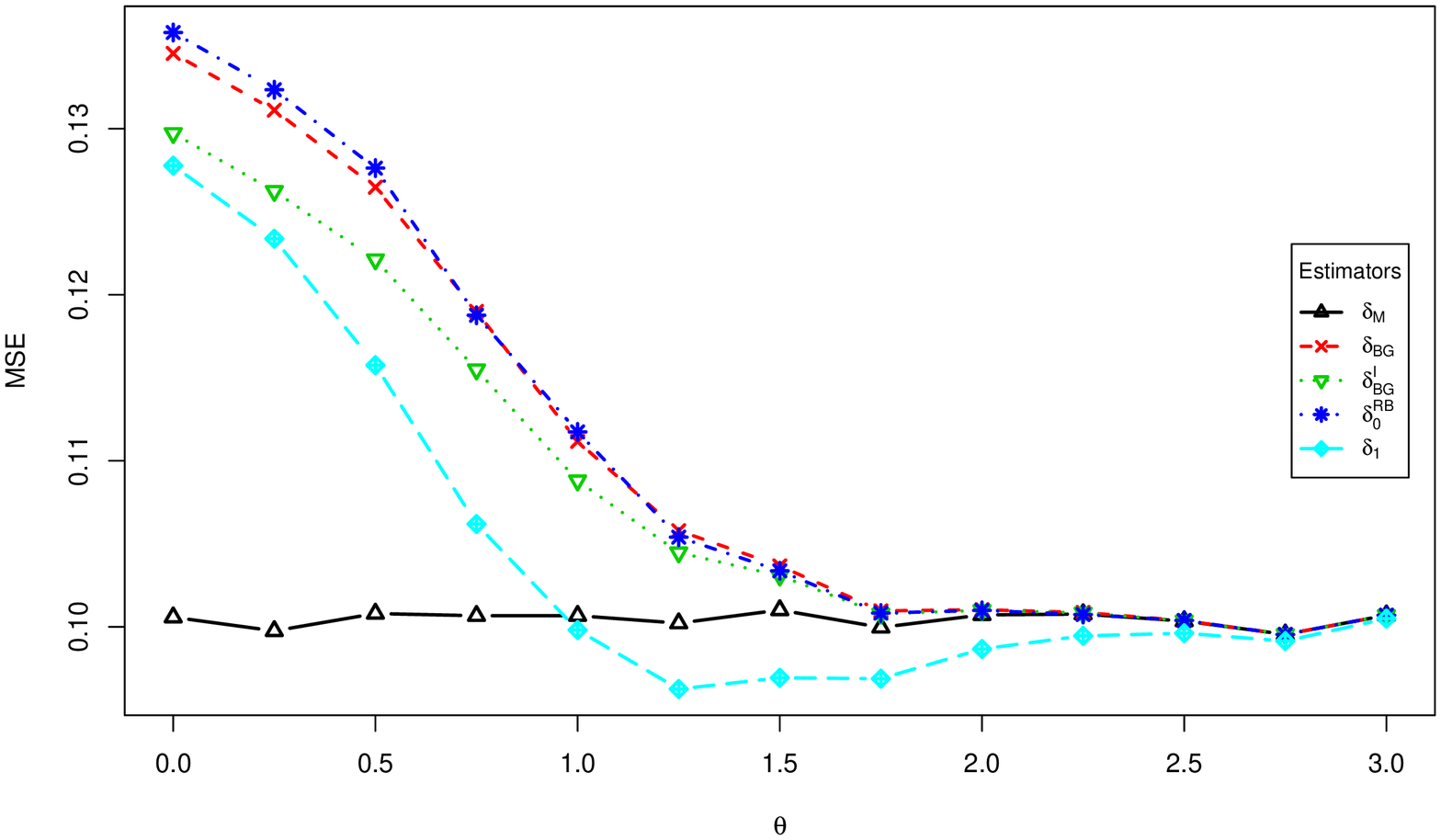}
	\caption{\textbf{Risk plots of different competing estimators for $n_1=5, n_2=5$}}
\end{figure}
\vskip -0.1in
\begin{figure}[!h]
	\centering
	\includegraphics[height=3.6in,width=4.6in]{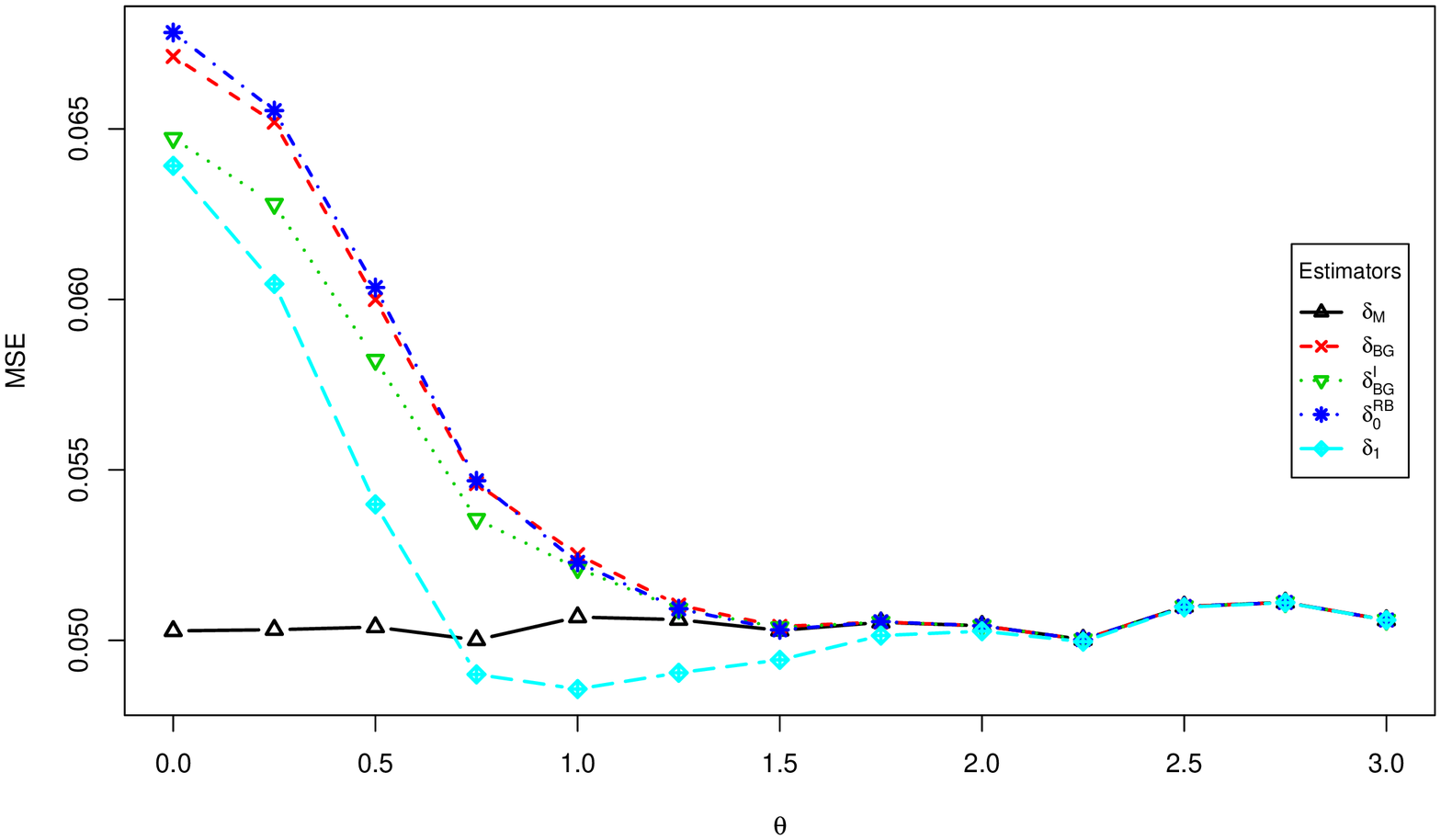}
	\caption{\textbf{Risk plots of different competing estimators for $n_1=10,n_2=10$.}}
\end{figure}
\FloatBarrier
\vskip -0.1in
\begin{figure}[!h]
	\centering
	\includegraphics[height=3.6in,width=4.3in]{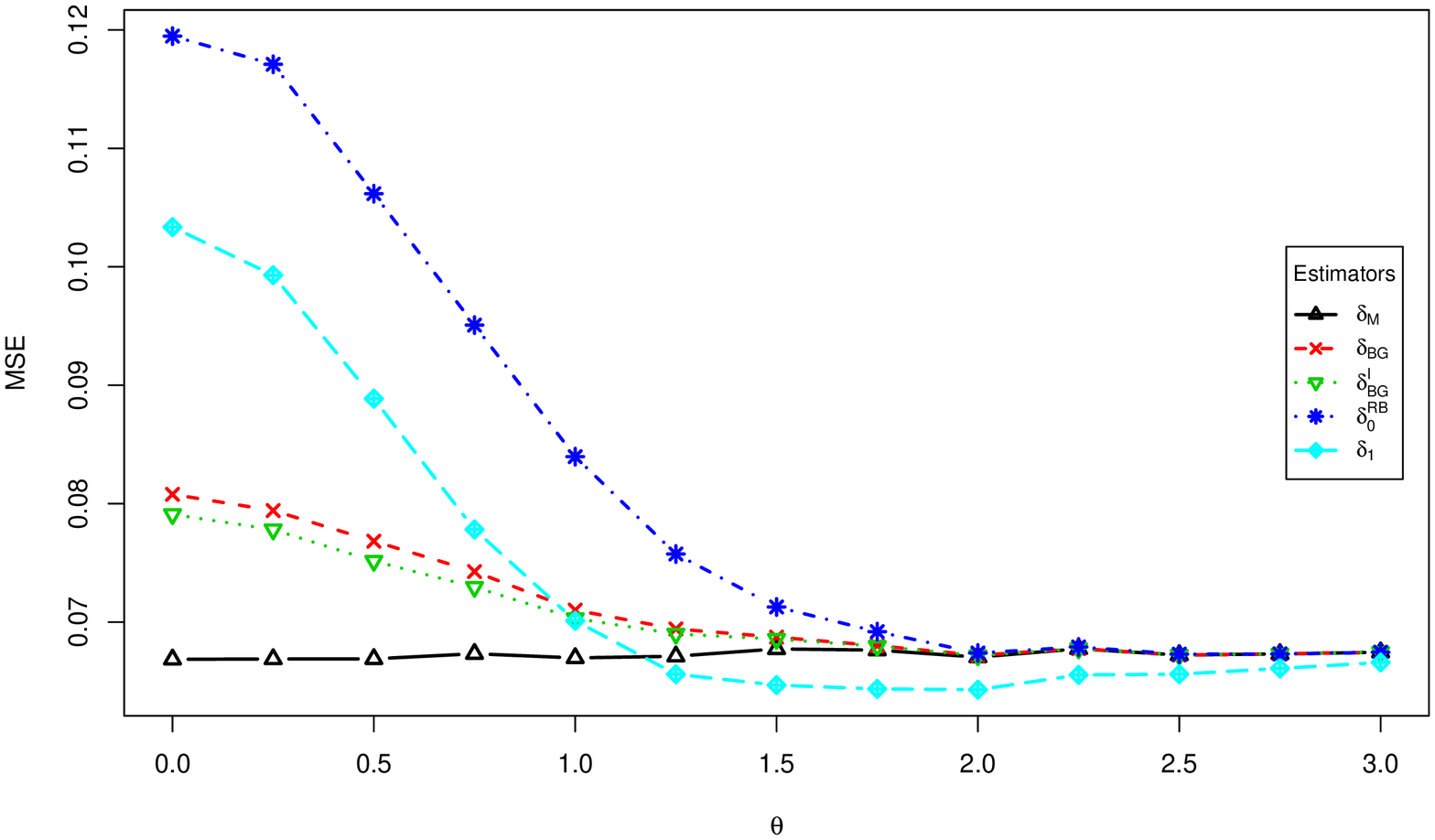}
	\caption{\textbf{Risk plots of different competing estimators for $n_1=5, n_2=10$}}
\end{figure}
\vskip -0.1in
\begin{figure}[!h]
	\centering
	\includegraphics[height=3.6in,width=4.3in]{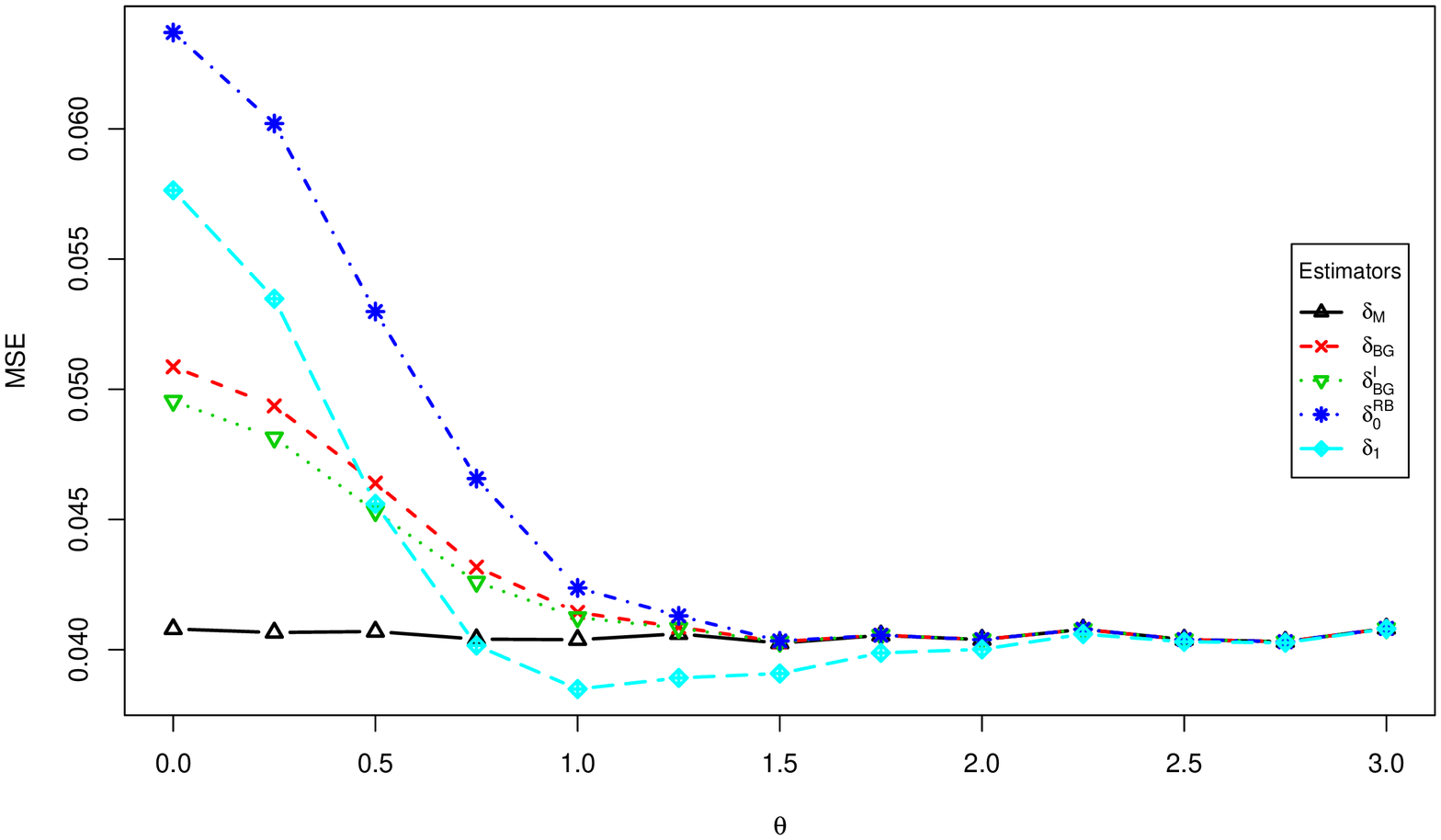}
	\caption{\textbf{Risk plots of different competing estimators for $n_1=10,n_2=15$.}}
\end{figure}
\FloatBarrier
\FloatBarrier
\vskip -0.1in
\begin{figure}[!h]
	\centering
	\includegraphics[height=3.6in,width=4.3in]{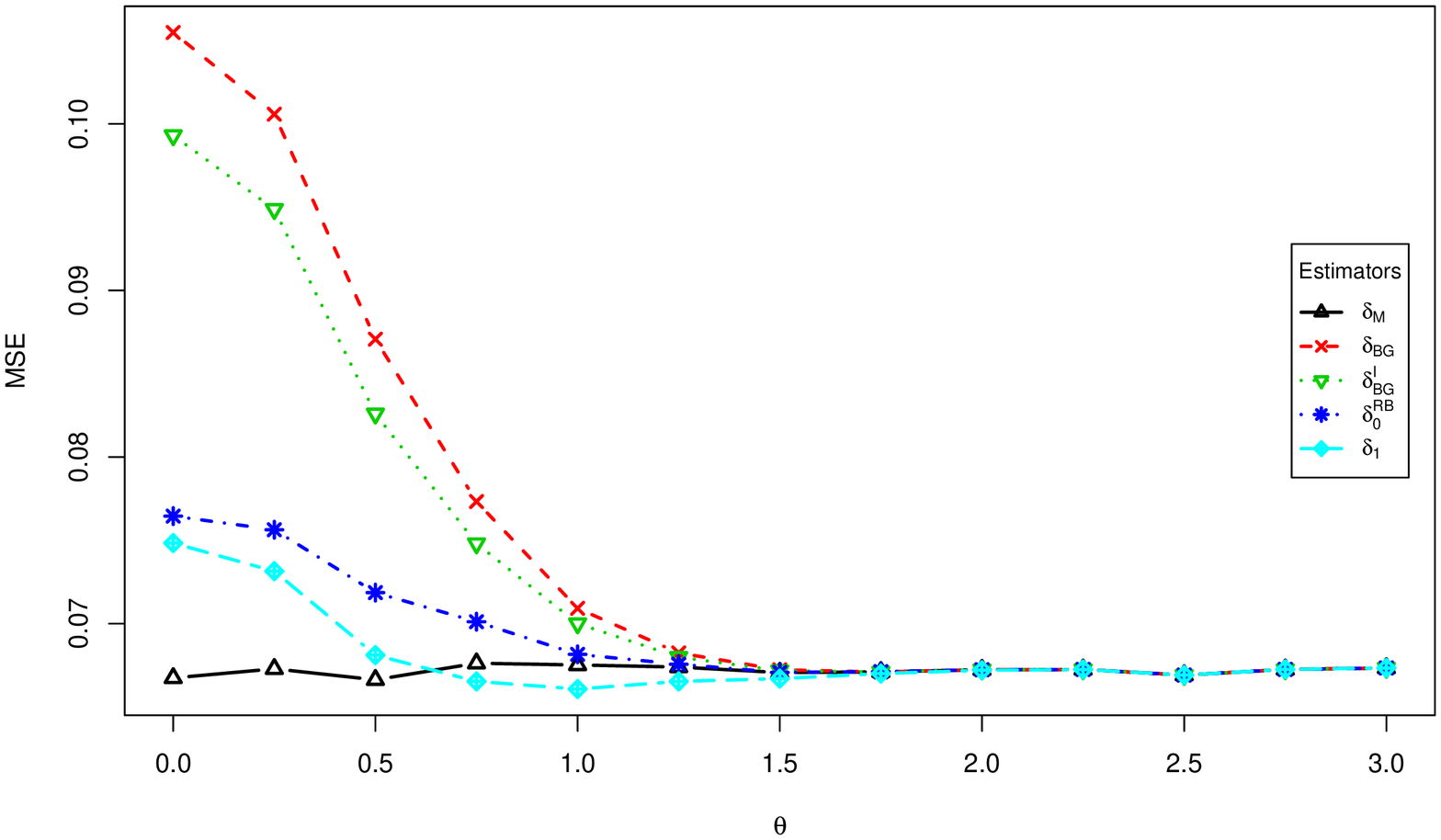}
	\caption{\textbf{Risk plots of different competing estimators for $n_1=10, n_2=5$}}
\end{figure}
\vskip -0.1in
\begin{figure}[!h]
	\centering
	\includegraphics[height=3.6in,width=4.3in]{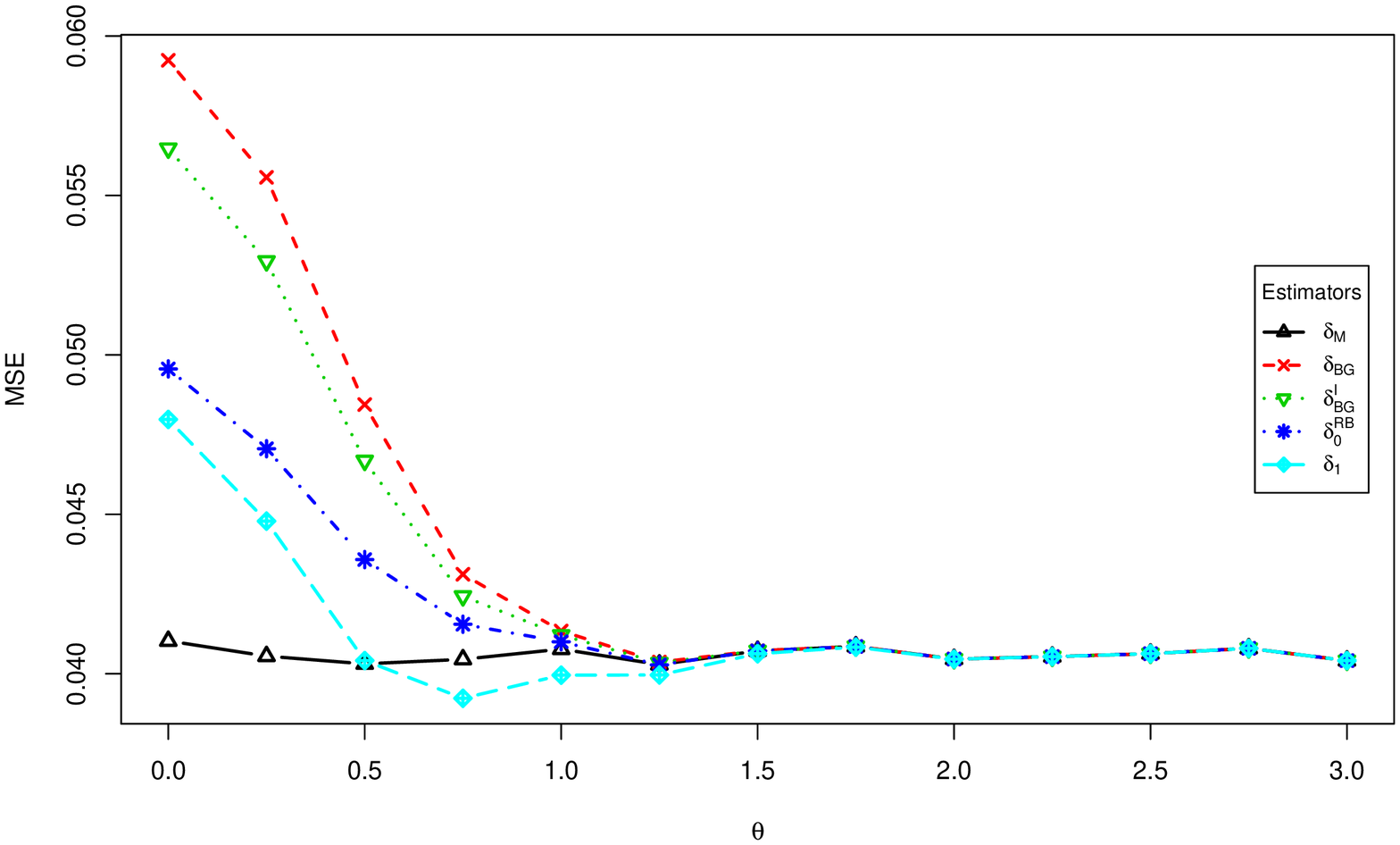}
	\caption{\textbf{Risk plots of different competing estimators for $n_1=15,n_2=10$.}}
\end{figure}


 \FloatBarrier
 \vskip -0.1in
 \begin{figure}[!h]
 	\centering
 	\includegraphics[height=3.6in,width=4.3in]{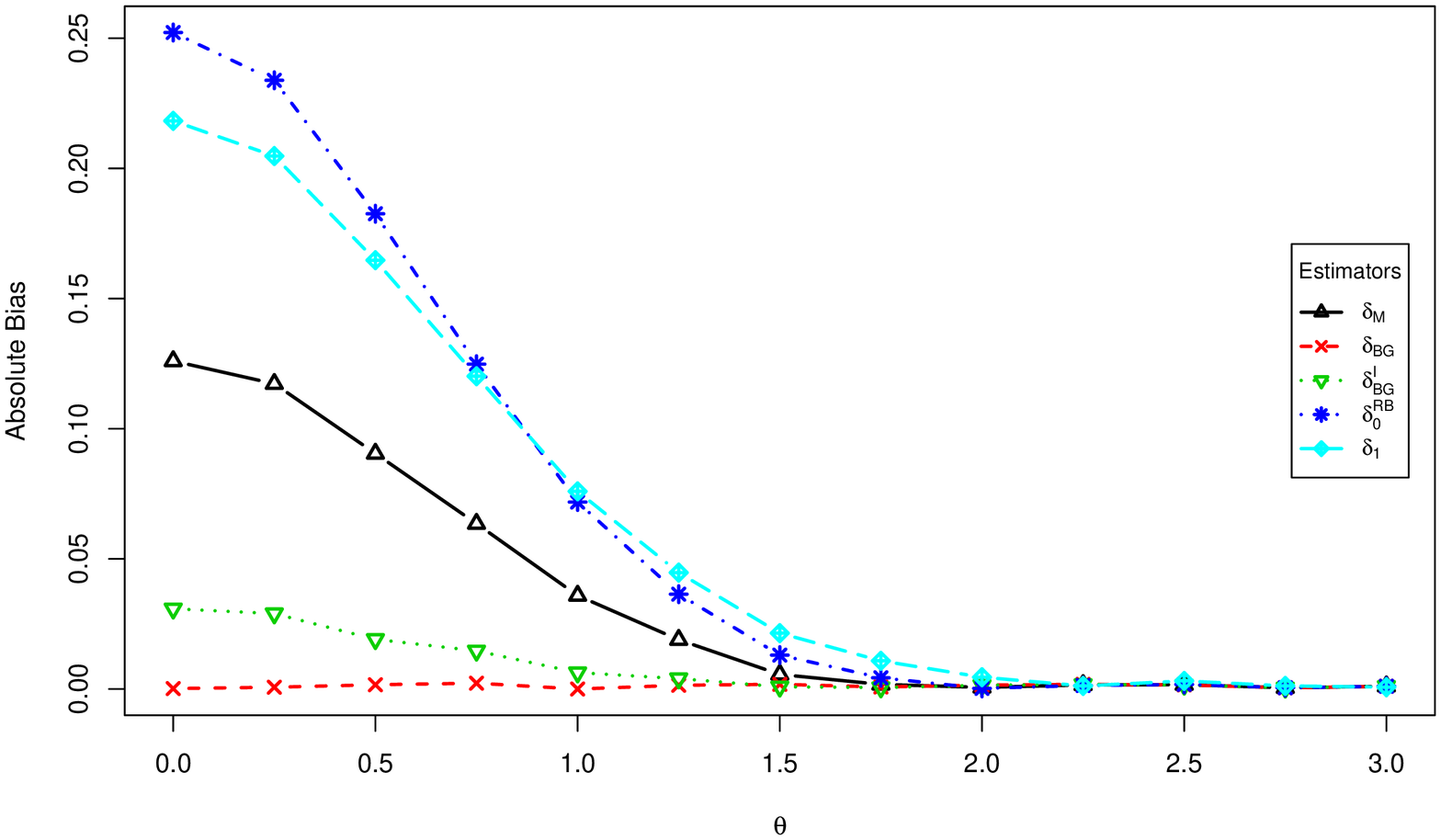}
 	\caption{\textbf{Bias plots of different competing estimators for $n_1=5, n_2=5$}}
 \end{figure}
 \vskip -0.1in
 \begin{figure}[!h]
 	\centering
 	\includegraphics[height=3.6in,width=4.3in]{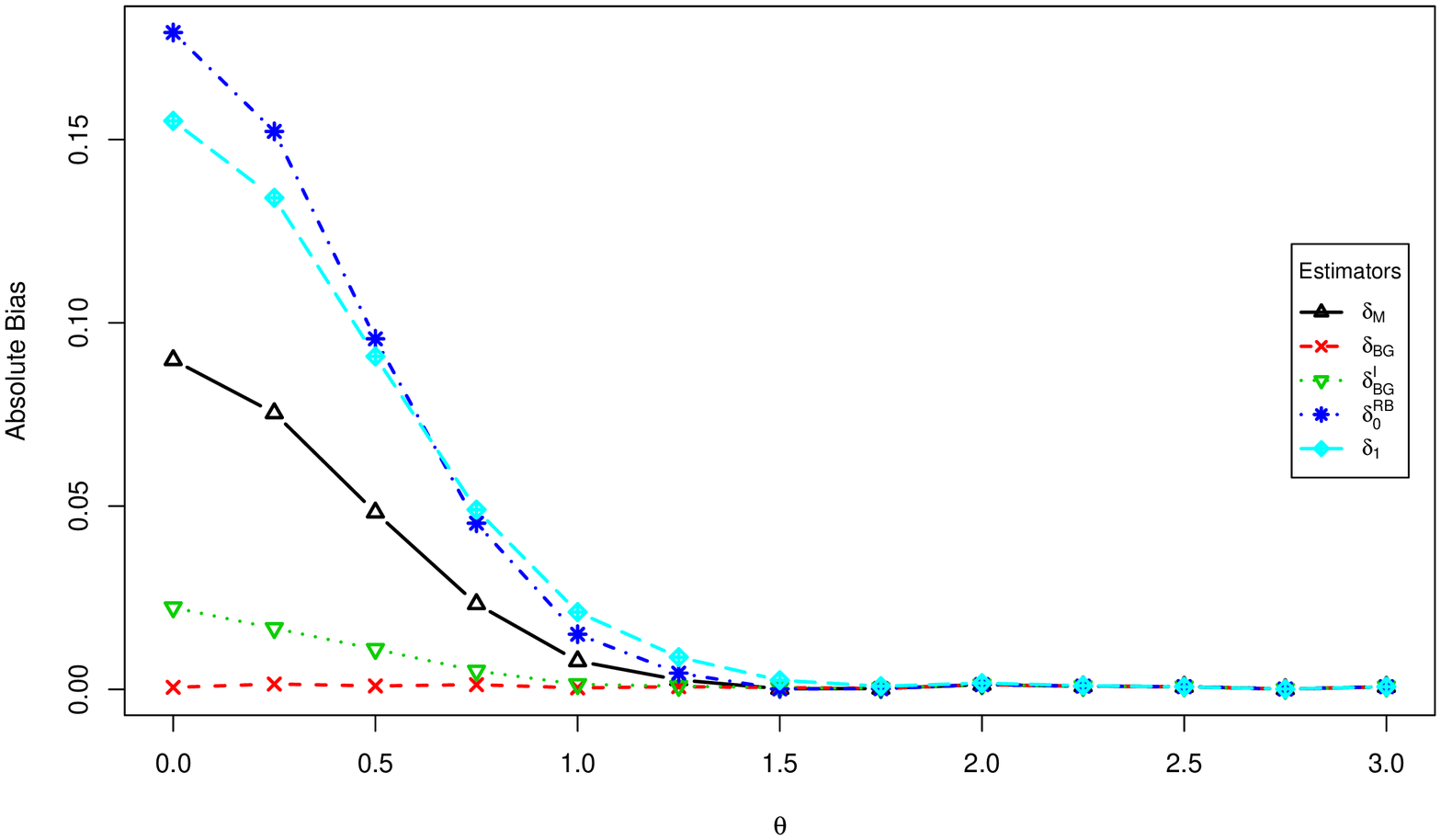}
 	\caption{\textbf{Bias plots of different competing estimators for $n_1=10,n_2=10$.}}
 \end{figure}
 \FloatBarrier
 \FloatBarrier
 \vskip -0.1in
 \begin{figure}[!h]
 	\centering
 	\includegraphics[height=3.6in,width=4.3in]{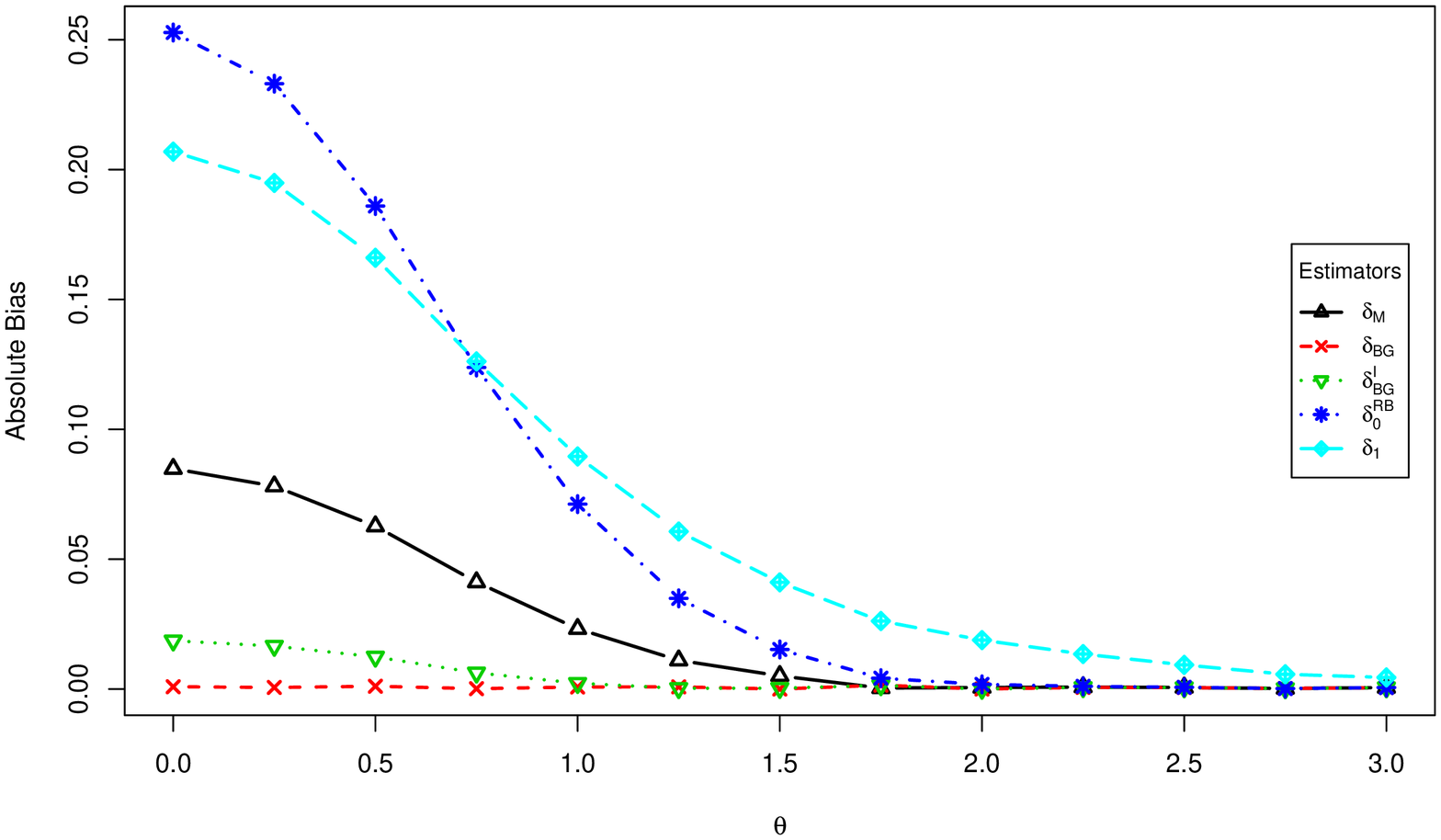}
 	\caption{\textbf{Bias plots of different competing estimators for $n_1=5, n_2=10$}}
 \end{figure}
 \vskip -0.1in
 \begin{figure}[!h]
 	\centering
 	\includegraphics[height=3.6in,width=4.3in]{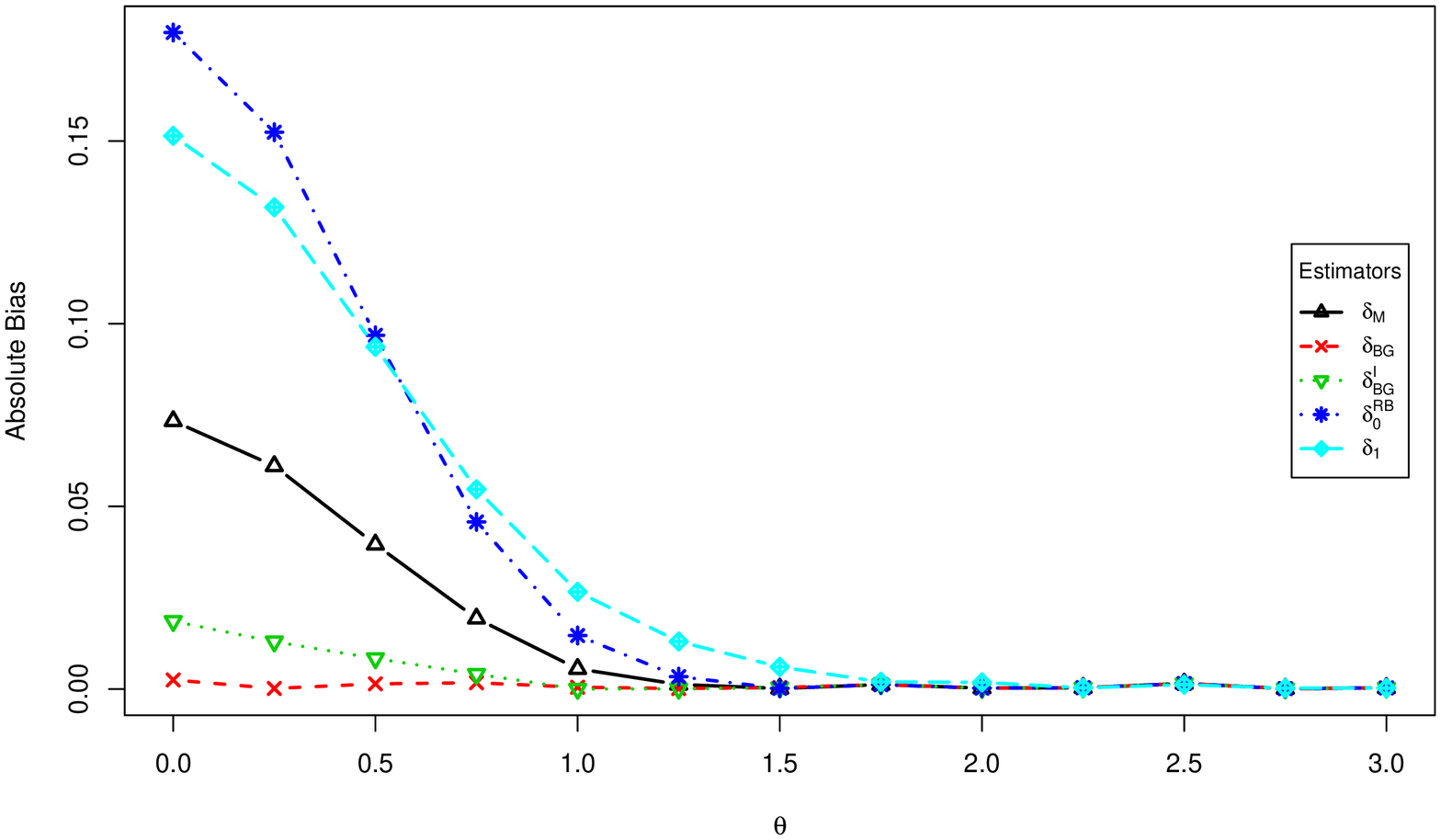}
 	\caption{\textbf{Bias plots of different competing estimators for $n_1=10,n_2=15$.}}
 \end{figure}
 \FloatBarrier
 \FloatBarrier
 \vskip -0.1in
 \begin{figure}[!h]
 	\centering
 	\includegraphics[height=3.6in,width=4.3in]{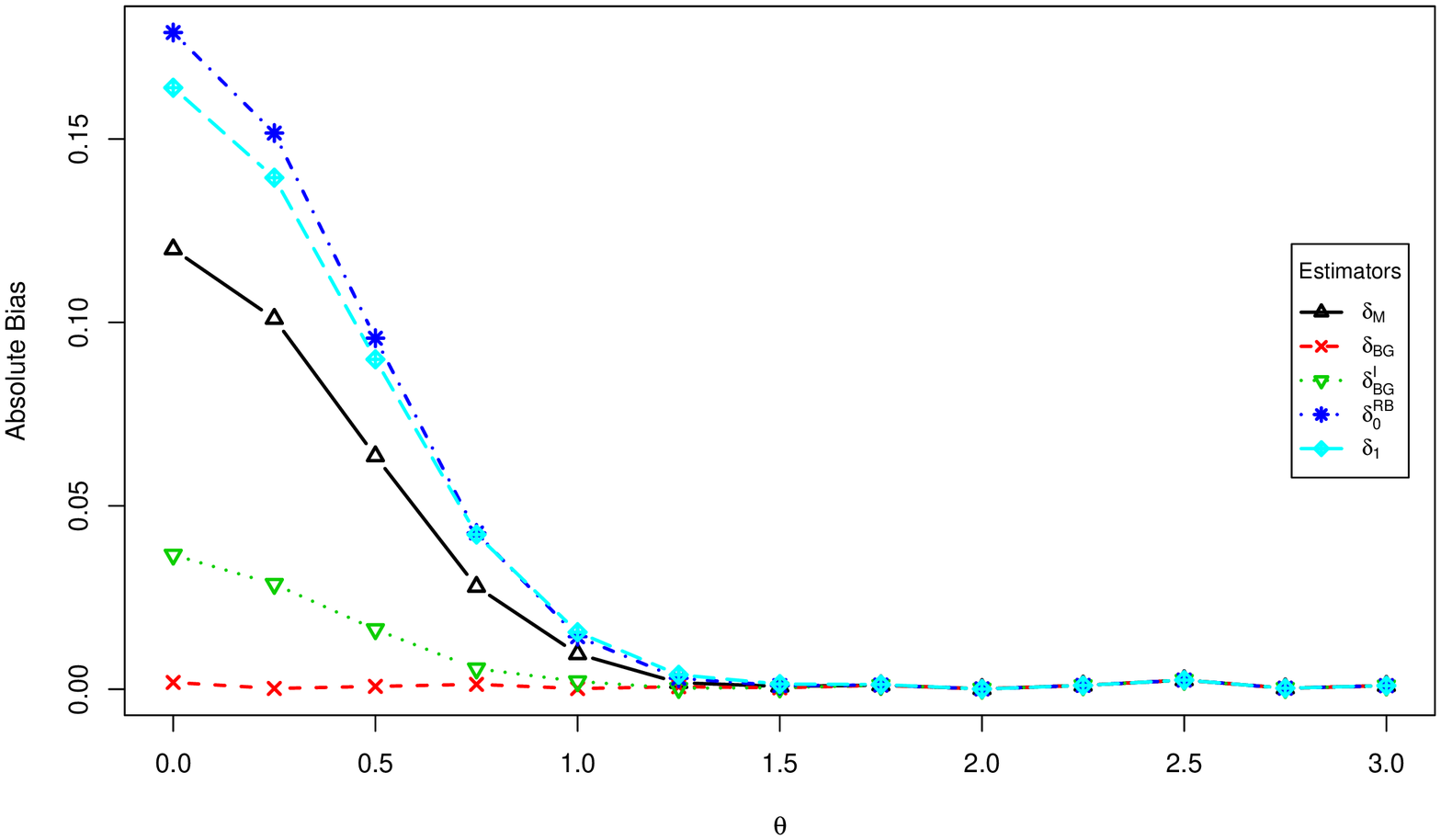}
 	\caption{\textbf{Bias plots of different competing estimators for $n_1=10, n_2=5$}}
 \end{figure}
 \vskip -0.1in
 \begin{figure}[!h]
 	\centering
 	\includegraphics[height=3.6in,width=4.3in]{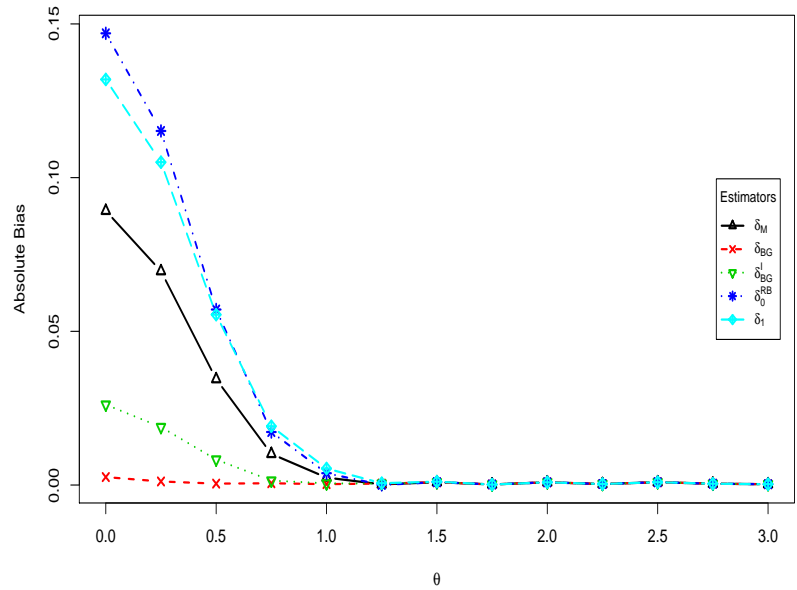}
 	\caption{\textbf{Bias plots of different competing estimators for $n_1=15,n_2=10$.}}
 \end{figure}
 \FloatBarrier
 \section{Real data example}
 In this section, we illustrate the implementation of the findings our  paper to a data set. For the purpose of illustration, we take a part of data set given in \cite{abdalghani2021adjusted} (one may also refer the book by \cite{daniel2018biostatistics}).  We have presented it in Table 2 below. The data reflects the serum concentrations (treatment effects) of binding protein-3 after use  of growth hormone (GH) and insulin-like growth factor I (IGF-I) on biochemical markers of bone metabolism in patients of idiopathic osteoporosis.
 Let $\Pi_1$ and $\Pi_2$ denote the populations receiving the treatments GH and IGF-I, respectively. As verified in \cite{abdalghani2021adjusted}, the  data are assumed to be normally distributed with different means and a common variance. 
 \FloatBarrier
 \begin{table}[h!]
 	\centering
 	\textbf{Table 2}\\
 	\textbf{Serum concentrations of binding protein-3 after the treatment GH.}\\\textbf{Stage I data}
 	\begin{tabular}{lllllllllll}
 		4507 & 4072 & 3036 &  & 2484 & 3540 &  & 3480 & 2055 &  &   \\
 		4095 & 2315 & 1840 &  & 2483 & 2354 &  & 3178 & 3574 &  &   \\
 		3196 & 2365 & 4136 &  & 3088 & 3464 &  & 5874 & 2929 &  &   \\
 		3903 & 3367 & 2938 &  & 4142 & 4465 &  & 3967 & 4213 &  &   \\
 		4321 & 4990 & 3622 &  & 6800 & 6185 &  & 4247 & 4450 &  &   \\
 		4199 & 5390 & 5188 &  & 4788 & 4602  
 	\end{tabular}
 \end{table}
 \begin{table}[h!]
 	\centering
 	\textbf{Serum concentrations of binding protein-3 after the use of  IGF-I treatment.}\\	\textbf{Stage I data}
 	\begin{tabular}{lllllllllll}
 		3480 & 3515 & 4003 & &3667 & 4263 & &4797 & 2354 & &   \\
 		3570 & 3630 & 3666 & &2700 & 2782 & &3088 & 3405  & &   \\
 		3309 & 3444 & 2357 & &3831 & 2905 & &2888 & 2797  &  &   \\
 		3083 & 3376 & 3464 & &4990 & 4590 & &2989 & 4081  &  &   \\
 		4806 & 4435 & 3504 & &3529 & 4093 & &4114 & 4445  &  &   \\
 		3622 & 5130 & 4784 & &4093 & 4852 	
 	\end{tabular}
 \end{table}
 \begin{table}[h!]
	\centering
	\textbf{Serum concentrations of binding protein-3 after the treatment GH.}\\\textbf{Stage II data}
	\begin{tabular}{lllllllllll}
		3161 & 4942 & 3222 &  & 2699 & 3514 &  & 2963 & 3228 &  &   \\
		5995 & 3315 & 2919 &  & 3235 & 4379 &  & 5628 & 6152 &  &   \\
		4415 & 5251 & 3334 &  & 3910 & 2304 &  & 4721 & 3700 &  &   \\
		3228 & 2440 & 2698 &  & 5793 & 4926   
	\end{tabular}
\end{table}
\FloatBarrier
 It is worth noting that the quality of the two treatments are assessed  in terms of their average effects. Therefore, the population corresponding the larger value of mean effect ($\max\left\{\mu_1,\mu_2\right\}$) is considered to be more effective. In stage I of the Drop-the-loser design, we first draw a sample of size $n_1$ from the two populations, and choose the population corresponding to larger sample mean $\left(\overline{X}_S=\max\left\{\overline{X}_1,\overline{X}_2\right\}\right)$ and drop the population corresponding to smaller sample mean effects $\left(\overline{X}_{3-S}=\min\left\{\overline{X}_1,\overline{X}_2\right\}\right)$. In stage II, we draw another sample of size $n_2$ from the selected population in stage I and name it as $\overline{Y}$. From the data given in Table 2, for $n_1=40$ and $n_2=26$, the observed average effects corresponding the GH and IGF-I treatments are obtained as: $\overline{X}_1=3846.05$, $\overline{X}_2=3710.775$. Further, we obtain $\overline{X}_S=3846.05$, $\overline{X}_{3-S}=3710.775$, $\overline{Y}=3925.846$, $D_1=-135.275 $, $D_2=79.796$. \par
 The various estimates of the selected treatment mean $\mu_{S}$ are tabulated in Table below:
   \begin{center}
 		\textbf{Table 3:} \textbf{Various estimates of the selected treatment mean $\mu_{S}$}.
 	\label{tab:table1}\\
 	\begin{tabular}{|l|c|c|c|c|c|c|} \hline
 	 \textbf{$\delta_{M}$} &  \textbf{$\delta_{BG}$} &\textbf{$\delta^I_{BG}$} & \textbf{$\delta_{0}$} & \textbf{$\delta^I_{0}$} & \textbf{$\delta^{RB}_{0}$} & \textbf{$\delta_{1}$} \\
 		  
 		\hline
 		 3877.484 & 3860.262& 3862.575 & 3846.05& 3848.575& 3850.142& 3857.382\\ \hline
 	
 	\end{tabular}
 \end{center}
Therefore, on the basis of above analysis, we conclude that the treatment GH is more effective as compared to the treatment IGF-I and it is recommended for future applications. After applying the treatment GH, if one prefers to use the estimator $\delta_{M}$ then it can be anticipated to have $3877.484$ units serum concentration of binding protein-3.

\section{Concluding Remarks}
In case of single stage sampling alone, estimation following selection of treatments is prone to bias, especially for normally distributed data. In the literature there are results, where it has been shown that no unbiased estimator of the selected treatment mean exists (see \cite{putter1968technical}, \cite{vellaisamy2009note} and \cite {masihuddin2021equivariant}). To overcome this issue, \cite{cohen1989two}, \cite{bowden2008unbiased} and \cite{robertson2019conditionally} have used the two stage adaptive design approach and provided two stage conditionally unbiased estimator. \par Under the criterion of mean squared error, in this paper, we have addressed the problem of efficient estimation of the selected treatment mean under two-stage drop the loser set up. Our main objective was to look for estimators of the selected treatment mean that perform better in terms of the mean squared error. In this direction we have shown that the maximum likelihood estimator, which is the weighted average of the first and second stage sample means (with weights being proportional to the corresponding sample sizes), is minimax and admissible. Under the two-stage DLD, a sufficient condition for inadmissibility of an arbitrary location and permutation equivariant estimator has been derived. As a consequence of which the two stage UMVCUE, proposed by \cite{bowden2008unbiased}, is shown to be inadmissible and a better estimator is obtained.\par 
It would be interesting to explore, if the proposed approach works well for non-normal data and involving more than two treatment groups. We will make attempts in these directions in our future works.

	



\bibliographystyle{apalike}
\bibliography{ref3}

\begin{thebibliography}{}

\bibitem[Abdalghani et~al., 2021]{abdalghani2021adjusted}
Abdalghani, O., Arshad, M., Meena, K., and Pathak, A. (2021).
\newblock Adjusted bias and risk for estimating treatment effect after
  selection with an application in idiopathic osteoporosis.
\newblock {\em Optimal Decision Making in Operations Research and Statistics:
  Methodologies and Applications}, page 370.

\bibitem[Bahadur and Goodman, 1952]{bahadur1952impartial}
Bahadur, R.~R. and Goodman, L.~A. (1952).
\newblock Impartial decision rules and sufficient statistics.
\newblock {\em The Annals of Mathematical Statistics}, pages 553--562.

\bibitem[Bauer et~al., 2016]{bauer2016twenty}
Bauer, P., Bretz, F., Dragalin, V., K{\"o}nig, F., and Wassmer, G. (2016).
\newblock Twenty-five years of confirmatory adaptive designs: opportunities and
  pitfalls.
\newblock {\em Statistics in Medicine}, 35(3):325--347.

\bibitem[Bauer and Kieser, 1999]{bauer1999combining}
Bauer, P. and Kieser, M. (1999).
\newblock Combining different phases in the development of medical treatments
  within a single trial.
\newblock {\em Statistics in medicine}, 18(14):1833--1848.

\bibitem[Bowden and Glimm, 2008]{bowden2008unbiased}
Bowden, J. and Glimm, E. (2008).
\newblock Unbiased estimation of selected treatment means in two-stage trials.
\newblock {\em Biometrical Journal: Journal of Mathematical Methods in
  Biosciences}, 50(4):515--527.

\bibitem[Brewster and Zidek, 1974]{brewster1974improving}
Brewster, J.-F. and Zidek, J. (1974).
\newblock Improving on equivariant estimators.
\newblock {\em The Annals of Statistics}, 2(1):21--38.

\bibitem[Cohen and Sackrowitz, 1989]{cohen1989two}
Cohen, A. and Sackrowitz, H.~B. (1989).
\newblock Two stage conditionally unbiased estimators of the selected mean.
\newblock {\em Statistics \& Probability Letters}, 8(3):273--278.

\bibitem[Dahiya, 1974]{dahiya1974estimation}
Dahiya, R.~C. (1974).
\newblock Estimation of the mean of the selected population.
\newblock {\em Journal of the American Statistical Association},
  69(345):226--230.

\bibitem[Daniel and Cross, 2018]{daniel2018biostatistics}
Daniel, W.~W. and Cross, C.~L. (2018).
\newblock {\em Biostatistics: A Foundation for Analysis in the Health
  Sciences}.
\newblock Wiley.

\bibitem[Eaton, 1967]{eaton1967some}
Eaton, M.~L. (1967).
\newblock Some optimum properties of ranking procedures.
\newblock {\em The Annals of Mathematical Statistics}, 38(1):124--137.

\bibitem[Hwang, 1993]{hwang1993empirical}
Hwang, J.~T. (1993).
\newblock Empirical bayes estimation for the means of the selected populations.
\newblock {\em Sankhy{\=a}: The Indian Journal of Statistics, Series A}, pages
  285--304.

\bibitem[Lu et~al., 2013]{lu2013estimating}
Lu, X., Sun, A., and Wu, S.~S. (2013).
\newblock On estimating the mean of the selected normal population in two-stage
  adaptive designs.
\newblock {\em Journal of Statistical Planning and Inference},
  143(7):1215--1220.

\bibitem[Masihuddin and Misra, 2021]{masihuddin2021equivariant}
Masihuddin and Misra, N. (2021).
\newblock Equivariant estimation following selection from two normal
  populations having common unknown variance.
\newblock {\em Statistics}, pages 1--32.

\bibitem[Misra and Dhariyal, 1994]{misra1994non}
Misra, N. and Dhariyal, I.~D. (1994).
\newblock Non-minimaxity of natural decision rules under heteroscedasticity.
\newblock {\em Statistics and Decisions}, (12):79--98.

\bibitem[Putter and Rubinstein, 1968]{putter1968technical}
Putter, J. and Rubinstein, D. (1968).
\newblock Technical report tr165: on estimating the mean of a selected
  population.
\newblock {\em University of Wisconsin statistics department, Wisconsin}.

\bibitem[Robertson et~al., 2021]{robertson2021point}
Robertson, D.~S., Choodari-Oskooei, B., Dimairo, M., Flight, L., Pallmann, P.,
  and Jaki, T. (2021).
\newblock Point estimation for adaptive trial designs.
\newblock {\em arXiv preprint arXiv:2105.08836}.

\bibitem[Robertson and Glimm, 2019]{robertson2019conditionally}
Robertson, D.~S. and Glimm, E. (2019).
\newblock Conditionally unbiased estimation in the normal setting with unknown
  variances.
\newblock {\em Communications in Statistics-Theory and Methods},
  48(3):616--627.

\bibitem[Sackrowitz and Samuel-Cahn, 1986]{sackrowitz1986evaluating}
Sackrowitz, H. and Samuel-Cahn, E. (1986).
\newblock Evaluating the chosen population: a bayes and minimax approach.
\newblock {\em Lecture Notes-Monograph Series}, pages 386--399.

\bibitem[Sampson and Sill, 2005]{sampson2005drop}
Sampson, A.~R. and Sill, M.~W. (2005).
\newblock Drop-the-losers design: normal case.
\newblock {\em Biometrical Journal: Journal of Mathematical Methods in
  Biosciences}, 47(3):257--268.

\bibitem[Sill and Sampson, 2007]{sill2007extension}
Sill, M.~W. and Sampson, A.~R. (2007).
\newblock Extension of a two-stage conditionally unbiased estimator of the
  selected population to the bivariate normal case.
\newblock {\em Communications in Statistics—Theory and Methods},
  36(4):801--813.

\bibitem[Tappin, 1992]{tappin1992unbiased}
Tappin, L. (1992).
\newblock Unbiased estimation of the parameter of a selected binomial
  population.
\newblock {\em Communications in Statistics-Theory and Methods},
  21(4):1067--1083.

\bibitem[Thall et~al., 1988]{thall1988two}
Thall, P.~F., Simon, R., and Ellenberg, S.~S. (1988).
\newblock Two-stage selection and testing designs for comparative clinical
  trials.
\newblock {\em Biometrika}, 75(2):303--310.

\bibitem[Thall et~al., 1989]{thall1989two}
Thall, P.~F., Simon, R., and Ellenberg, S.~S. (1989).
\newblock A two-stage design for choosing among several experimental treatments
  and a control in clinical trials.
\newblock {\em Biometrics}, pages 537--547.

\bibitem[Vellaisamy, 2009]{vellaisamy2009note}
Vellaisamy, P. (2009).
\newblock A note on unbiased estimation following selection.
\newblock {\em Statistical Methodology}, 6(4):389--396.

\bibitem[Wu et~al., 2010]{wu2010interval}
Wu, S.~S., Wang, W., and Yang, M.~C. (2010).
\newblock Interval estimation for drop-the-losers designs.
\newblock {\em Biometrika}, 97(2):405--418.

\end{thebibliography}
\section*{Appendix}
\textbf{Proof of Lemma 2.2.}: $(i)$ Let us introduce the auxiliary random variables $\overline{Y}_1$ and $\overline{Y}_2$ such that $\overline{X}_1$, $\overline{X}_2$, $\overline{Y}_1$ and $\overline{Y}_2$ are mutually independent and $\overline{Y}_i \sim N\left(\mu_i,\frac{\sigma^2}{n_2}\right),~ i=1,2$. Let $V_1=\frac{\sqrt{n_1}}{\sigma}\left(\overline{X}_1-\mu_{1}\right)$, $V_2=\frac{\sqrt{n_1}}{\sigma}\left(\overline{X}_2-\mu_{2}\right)$, $V_3=\frac{\sqrt{n_2}}{\sigma}\left(\overline{Y}_1-\mu_{1}\right)$, $V_4=\frac{\sqrt{n_2}}{\sigma}\left(\overline{Y}_2-\mu_{2}\right)$ and $w=\frac{n_1}{n_1+n_2}$.  Then, the  c.d.f. of       $U=\frac{n_1\overline{X}_S+n_2\overline{Y}}{n_1+n_2}-\mu_S=w\overline{X}_S+(1-w)\overline{Y} -\mu_{S}$ is
\begin{align}
F_{2,\underline{\mu}}(u)&=\mathbb{P}_{\underline{\mu}}\left(w\overline{X}_S+(1-w)\overline{Y}-\mu_{S} \leq u\right) \nonumber \\
&=\mathbb{P}_{\underline{\mu}}\left(\overline{X}_1 < \overline{X}_2, w\overline{X}_2+(1-w)\overline{Y}_2-\mu_{2} \leq u\right) +\mathbb{P}_{\underline{\mu}}\left( \overline{X}_2 \leq \overline{X}_1, w\overline{X}_1+(1-w)\overline{Y}_1-\mu_{1} \leq u\right) \nonumber\\
&=h_u(\mu_{1},\mu_{2})+h_u(\mu_{2},\mu_{1}),~ -\infty < u < \infty.
\end{align}
Due to symmetry we may assume that $\mu_1 \leq \mu_2$, so that $\mu_2-\mu_1=\theta.$ We have, for $-\infty <u <\infty$
\begin{align}
h_u(\mu_{1},\mu_{2})&=\mathbb{P}_{\underline{\mu}}\left(\overline{X}_1 < \overline{X}_2, w\overline{X}_2+(1-w)\overline{Y}_2-\mu_{2} \leq u\right) \nonumber\\
&=\mathbb{P}_{\underline{\mu}}\left(V_1-V_2 < \frac{\sqrt{n_1}}{\sigma}(\mu_2-\mu_{1}), \frac{w \sigma}{\sqrt{n_1}}V_2+\frac{(1-w) \sigma}{\sqrt{n_2}} V_4 \leq u \right) \nonumber\\
&= \mathbb{P}_{\underline{\mu}}\left(B_1 \leq \frac{\sqrt{n_1}}{\sigma}\theta, B_2 \leq u \right), \theta \geq 0,
\end{align}
where, $B_1=V_1-V_2 \sim N(0,2)$, $B_2=\frac{w \sigma}{\sqrt{n}_1}V_2+\frac{(1-w) \sigma}{\sqrt{n_2}}V_4 \sim N(0,\sigma^2_{*})$, and
$$ \begin{bmatrix}B_1 \\ B_2 \end{bmatrix} \sim N_2\left(\begin{bmatrix}0 \\ 0 \end{bmatrix} , \begin{bmatrix}2 & -\frac{w \sigma}{\sqrt{n_1}} \\ -\frac{w \sigma}{\sqrt{n_1}} & \sigma^2_{*} \end{bmatrix} \right).$$
Then, the conditional distribution of $B_1$ given that $B_2=t~ (t \in \mathbb{R})$, follows $N\left(-\frac{t\sqrt{n_1}}{\sigma}, 1+\rho\right)$ where, $\rho=\frac{n_2}{n_1+n_2}$.
Therefore, $(7.2)$ becomes
\begin{align}
h_u(\mu_{1},\mu_{2})&=
\displaystyle{\int_{-\infty}^{u}} \Phi\left(\frac{\sqrt{n_1}(t+\theta)}{\sigma\sqrt{1+\rho}}\right)\frac{1}{\sigma_{*}} \phi\left(\frac{t}{\sigma_{*}}\right)dt,~ -\infty < u <\infty
\end{align}
 and, by symmetry, we get
\begin{align}
h_u(\mu_{2},\mu_{1})
&=\displaystyle{\int_{-\infty}^{u}} \Phi\left(\frac{\sqrt{n_1}(t-\theta)}{\sigma\sqrt{1+\rho}}\right)\frac{1}{\sigma_{*}} \phi\left(\frac{t}{\sigma_{*}}\right)dt,~ -\infty < u <\infty.
\end{align}
Consequently, using (7.1), (7.3) and (7.4), the c.d.f. of $U$ is
\begin{align}
F_{2,\underline{\mu}}(u)&=\displaystyle{\int_{-\infty}^{u}} \left[\Phi\left(\frac{\sqrt{n_1}(t+\theta)}{\sigma\sqrt{1+\rho}}\right)+\Phi\left(\frac{\sqrt{n_1}(t-\theta)}{\sigma\sqrt{1+\rho}}\right)\right]\frac{1}{\sigma_{*}} \phi\left(\frac{t}{\sigma_{*}}\right)dt,~ -\infty < u <\infty.
\end{align}
Hence, the assertion follows.\\

$(ii)$ Let $A=\frac{\sqrt{n_1}\sigma_{*}}{\sigma\sqrt{1+\rho}}$ and $B=\frac{\sqrt{n_1}\theta}{\sigma\sqrt{1+\rho}}$. For any $ \underline{\mu} \in \Theta $, using Lemma $2.1(iii)$we have,
\begin{align*}
\mathbb{E}_{\underline{\mu}}\left(U^2\right)&=\displaystyle{\int_{-\infty}^{\infty}}  \frac{u^2}{\sigma_{*}}\left[\Phi\left(\frac{\sqrt{n_1}(u+\theta)}{\sigma\sqrt{1+\rho}}\right)+\Phi\left(\frac{\sqrt{n_1}(u-\theta)}{\sigma\sqrt{1+\rho}}\right)\right]\phi\left(\frac{u}{\sigma_{*}}\right)du\\
&=\sigma^2_{*}\left[\displaystyle{\int_{-\infty}^{\infty}} y^2 \Phi\left(Ay+B\right)\phi(y)dy+ \displaystyle{\int_{-\infty}^{\infty}} y^2 \Phi\left(Ay-B\right)\phi(y)dy\right] \\
&=\sigma^2_{*}.
\end{align*}

\textbf{Proof of Lemma 4.1.}: (a)	Let $d_1 \in (-\infty,0],~  d_2 \in \mathbb{R}$ be fixed. Then, the conditional c.d.f. of $S_1$ given $(D_1,D_2)=(d_1,d_2)$ is given by
\begin{align}
F_{1, \underline{\mu}}(s|d_1,d_2)&= \lim\limits_{(h,k) \downarrow (0,0)} \frac{\mathbb{P}_{\underline{\mu}}\left(S_1 \leq s, d_1-h< D_1 \leq d_1, d_2-k < D_2 \leq d_2 \right)}{\mathbb{P}_{\underline{\mu}}\left(d_1-h< D_1 \leq d_1, d_2-k < D_2 \leq d_2 \right)} \nonumber\\ 
&=\lim\limits_{(h,k) \downarrow (0,0)} \frac{N_1(h,k)}{N_2(h,k)}, \text{(say)}.
\end{align}
Let $\overline{Y}_1$, $\overline{Y}_2$, $V_1$, $V_2$, $V_3$ and $V_4$ be as defined in the proof of Lemma 2.2. Also, due to symmetry, we may assume that $\mu_{1} \leq \mu_{2}$, so that $\theta=\mu_{2}-\mu_{1}$. Then, for sufficiently sufficiently small $h>0, k>0$,
\begin{align*}
&N_1\left(h,k\right)\\&=\mathbb{P}_{\underline{\mu}}\left(S_1 \leq s, d_1-h< D_1 \leq d_1, d_2-k < D_2 \leq d_2 \right)\\
&=\mathbb{P}_{\underline{\mu}}\left(\overline{X}_S - \mu_{S} \leq s, d_1-h< \overline{X}_{3-S}-\overline{X}_S \leq d_1, d_2-k < \overline{Y}-\overline{X}_S \leq d_2 \right)\\
&=\mathbb{P}_{\underline{\mu}}\left( \overline{X}_1 > \overline{X}_2,  \overline{X}_1- \mu_1 \leq s, d_1-h<\overline{X}_2-\overline{X}_1 \leq d_1,d_2-k <\overline{Y}_1-\overline{X}_1 \leq d_2\right)\\
& ~\hspace{2mm}+ \mathbb{P}_{\underline{\mu}}\left( \overline{X}_2 \geq \overline{X}_1,  \overline{X}_2- \mu_2 \leq s, d_1-h<\overline{X}_1-\overline{X}_2 \leq d_1,d_2-k <\overline{Y}_2-\overline{X}_2 \leq d_2\right) \\
&=\mathbb{P}_{\underline{\mu}}\left( \overline{X}_1- \mu_1 \leq s, d_1-h<\overline{X}_2-\overline{X}_1 \leq d_1, d_2-k <\overline{Y}_1-\overline{X}_1 \leq d_2\right) \nonumber\\
& ~\hspace{2mm}+ \mathbb{P}_{\underline{\mu}}\left(\overline{X}_2- \mu_2 \leq s, d_1-h<\overline{X}_1-\overline{X}_2 \leq d_1,d_2-k <\overline{Y}_2-\overline{X}_2 \leq d_2\right) \nonumber\\
&=\mathbb{P}_{\underline{\mu}}\left( V_1 \leq \frac{\sqrt{n_1}}{\sigma}s, V_1+\frac{\sqrt{n_1}}{\sigma}\left(d_1-h-\theta \right) < V_2 \leq V_1+\frac{\sqrt{n_1}}{\sigma}\left(d_1-\theta \right),\frac{\sqrt{n_2}}{\sigma}\left(\frac{\sigma}{\sqrt{n_1}}V_1+d_2-k\right) < V_3 \leq \frac{\sqrt{n_2}}{\sigma}\left(\frac{\sigma}{\sqrt{n_1}}V_1+d_2\right)\right) \nonumber\\
&+\mathbb{P}_{\underline{\mu}}\left( V_2 \leq \frac{\sqrt{n_1}}{\sigma}s, V_2+\frac{\sqrt{n_1}}{\sigma}\left(d_1-h+\theta\right) < V_1 \leq V_2+\frac{\sqrt{n_1}}{\sigma}\left(d_1-\theta \right),\frac{\sqrt{n_2}}{\sigma}(\frac{\sigma}{\sqrt{n_1}}V_2+d_2-k)< V_4 \leq \frac{\sqrt{n_2}}{\sigma}\left(\frac{\sigma}{\sqrt{n_1}}V_2+d_2\right)\right) \nonumber\\
&=\displaystyle{\int_{-\infty}^{\frac{\sqrt{n_1}}{\sigma}s}}\left[\Phi\left(v+\frac{\sqrt{n_1}}{\sigma}\left(d_1-\theta\right)\right)-\Phi\left(v+\frac{\sqrt{n_1}}{\sigma}\left(d_1-h-\theta\right)\right)\right] \nonumber\\
&~~~~~~~~~~~~~~~~~~~~~~~~~~~~~~~~~~~~\times \left[\Phi\left(\frac{\sqrt{n_2}}{\sigma}(\frac{\sigma}{\sqrt{n_1}}v+d_2)\right)-\Phi\left(\frac{\sqrt{n_2}}{\sigma}(\frac{\sigma}{\sqrt{n_1}}v+d_2-k)\right)\right] \phi\left(v\right) dv  \nonumber \\
&~~~~~~~+ \displaystyle{\int_{-\infty}^{\frac{\sqrt{n_1}}{\sigma}s}}\left[\Phi\left(v+\frac{\sqrt{n_1}}{\sigma}\left(d_1+\theta\right)\right)-\Phi\left(v+\frac{\sqrt{n_1}}{\sigma}\left(d_1-h+\theta\right)\right)\right] \nonumber\\
&~~~~~~~~~~~~~~~~~~~~~~~~~~~~~~~~~~~~\times \left[\Phi\left(\frac{\sqrt{n_2}}{\sigma}(\frac{\sigma}{\sqrt{n_1}}v+d_2)\right)-\Phi\left(\frac{\sqrt{n_2}}{\sigma}(\frac{\sigma}{\sqrt{n_1}}v+d_2-k)\right)\right] \phi\left(v\right) dv.
\end{align*}
Similarly,
\begin{align}
N_2\left(h,k\right)&= \displaystyle{\int_{-\infty}^{\infty}}\left[\Phi\left(t+\frac{\sqrt{n_1}}{\sigma}\left(d_1-\theta\right)\right)-\Phi\left(t+\frac{\sqrt{n_1}}{\sigma}\left(d_1-h-\theta\right)\right)\right] \nonumber\\
&~~~~~~~~~~~~~~~~~~~~~~~~~~~~~~~~~~~~\times \left[\Phi\left(\frac{\sqrt{n_2}}{\sigma}(\frac{\sigma}{\sqrt{n_1}}t+d_2)\right)-\Phi\left(\frac{\sqrt{n_2}}{\sigma}(\frac{\sigma}{n_1}t+d_2-k)\right)\right] \phi\left(t\right) dt  \nonumber \\
&~~~~~~~+ \displaystyle{\int_{-\infty}^{\infty}}\left[\Phi\left(t+\frac{\sqrt{n_1}}{\sigma}\left(d_1+\theta\right)\right)-\Phi\left(t+\frac{\sqrt{n_1}}{\sigma}\left(d_1-h+\theta\right)\right)\right] \nonumber\\
&~~~~~~~~~~~~~~~~~~~~~~~~~~~~~~~~~~~~\times \left[\Phi\left(\frac{\sqrt{n_2}}{\sigma}(\frac{\sigma}{\sqrt{n_1}}t+d_2)\right)-\Phi\left(\frac{\sqrt{n_2}}{\sigma}(\frac{\sigma}{\sqrt{n_1}}t+d_2-k)\right)\right] \phi\left(t\right) dt.
\end{align}
Using $(7.6)$-$(7.8)$, and the L' Hopital rule, we get for fixed $d_1 \in (-\infty,0]$ and $d_2 \in \mathbb{R}$,
\begin{align}
F_{1,\underline{\mu}}(s|d_1,d_2)&=\frac{\displaystyle{\int_{-\infty}^{\frac{\sqrt{n_1}}{\sigma}s}}\left[\phi\left(v+\frac{\sqrt{n_1}}{\sigma}(d_1-\theta)\right)+\phi\left(v+\frac{\sqrt{n_1}}{\sigma}(d_1+\theta)\right)\right]\phi\left(\frac{\sqrt{n_2}}{\sigma}(\frac{\sigma}{\sqrt{n_1}}v+d_2)\right)\phi\left(v\right)dv}{\displaystyle{\int_{-\infty}^{\infty}}\left[\phi\left(t+\frac{\sqrt{n_1}}{\sigma}(d_1-\theta)\right)+\phi\left(t+\frac{\sqrt{n_1}}{\sigma}(d_1+\theta)\right)\right]\phi\left(\frac{\sqrt{n_2}}{\sigma}(\frac{\sigma}{\sqrt{n_1}}t+d_2)\right)\phi\left(t\right)dv},~s \in \mathbb{R}.
\end{align}
Consequently, for any fixed $d_1 \in (-\infty,0]$ and $d_2 \in \mathbb{R}$, the conditional p.d.f. of $S_1$ given $\left(D_1,D_2\right)=(d_1,d_2)$, is given by
\begin{align*}
f_{1, \underline{\mu}}(s|d_1,d_2)&= \frac{\left[\phi\left(\frac{\sqrt{n_1}}{\sigma}(s+d_1-\theta)\right)+\phi\left(\frac{\sqrt{n_1}}{\sigma}(s+d_1+\theta)\right)\right]\phi\left(\frac{\sqrt{n_2}}{\sigma}(s+d_2)\right)\phi\left(\frac{\sqrt{n_1}}{\sigma}s\right)}{\displaystyle{\int_{-\infty}^{\infty}}\left[\phi\left(\frac{\sqrt{n_1}}{\sigma}(t+d_1-\theta)\right)+\phi\left(\frac{\sqrt{n_1}}{\sigma}(t+d_1+\theta)\right)\right]\phi\left(\frac{\sqrt{n_2}}{\sigma}(t+d_2)\right)\phi\left(\frac{\sqrt{n_1}}{\sigma}t\right) dt } ~,s\in \mathbb{R}.\\
&=\frac{  \splitfrac{ e^{\frac{n_1}{(2n_1+n_2)\sigma^2}((n_1+n_2)d_1-n_2d_2)\theta}\frac{\sqrt{2n_1+n_2}}{\sigma}\phi\left(\frac{\sqrt{2n_1+n_2}}{\sigma}\left(s+\frac{n_1d_1+n_2d_2-n_1\theta}{2n_1+n_2}\right)\right)}{+ e^{-\frac{n_1}{(2n_1+n_2)\sigma^2}((n_1+n_2)d_1-n_2d_2)\theta} \frac{\sqrt{2n_1+n_2}}{\sigma}\phi\left(\frac{\sqrt{2n_1+n_2}}{\sigma}\left(s+\frac{n_1d_1+n_2d_2+n_1\theta}{2n_1+n_2}\right)\right)} }{e^{\frac{n_1}{(2n_1+n_2)\sigma^2}((n_1+n_2)d_1-n_2d_2)\theta}+e^{-\frac{n_1}{(2n_1+n_2)\sigma^2}((n_1+n_2)d_1-n_2d_2)\theta}},~ -\infty < s < \infty.
\end{align*}
(b) Using (a), for any fixed $d_1 \in (-\infty,0]$ and $d_2 \in \mathbb{R}$,
\begin{align}
&\mathbb{E}_{\underline{\mu}}(S_1|(D_1,D_2)=(d_1,d_2)) \nonumber\\
&=\frac{e^{\frac{n_1}{(2n_1+n_2)\sigma^2}((n_1+n_2)d_1-n_2d_2)\theta} \left(\frac{n_1\theta-n_1d_1-n_2d_2}{2n_1+n_2}\right)-e^{-\frac{n_1}{(2n_1+n_2)\sigma^2}((n_1+n_2)d_1-n_2d_2)\theta} \left(\frac{n_1\theta+n_1d_1+n_2d_2}{2n_1+n_2}\right)}{e^{\frac{n_1}{(2n_1+n_2)\sigma^2}((n_1+n_2)d_1-n_2d_2)\theta}+e^{-\frac{n_1}{(2n_1+n_2)\sigma^2}((n_1+n_2)d_1-n_2d_2)\theta}}\nonumber \\
&=\frac{n_1 \theta}{2n_1+n_2} \left\{\frac{1-e^{-\frac{2 \theta}{\sigma^2}\left(\frac{n_1((n_1+n_2)d_1-n_2d_2)}{2n_1+n_2}\right)}}{1+e^{-\frac{2 \theta}{\sigma^2}\left(\frac{n_1((n_1+n_2)d_1-n_2d_2)}{2n_1+n_2}\right)}}\right\}-\frac{n_1d_1+n_2d_2}{2n_1+n_2}.
\end{align}

\end{document}